\documentclass[reqno,12pt]{amsart} 

\usepackage{amssymb,amsmath,bbm}

\usepackage{tikz}
\usetikzlibrary{matrix}
\usepackage{float}
\usepackage[font=footnotesize]{caption}
\usetikzlibrary{arrows.meta}
\usetikzlibrary{decorations.markings}

\usepackage{esint}
\usepackage{bookmark}

\usepackage[left=1.2in, right=1.2in, bottom=1.5in]{geometry} 
\newtheorem{theorem}{Theorem}[section]
\newtheorem{lemma}[theorem]{Lemma}

\newtheorem{corollary}[theorem]{Corollary}
\theoremstyle{definition}
\newtheorem{definition}[theorem]{Definition}

\theoremstyle{remark}
\newtheorem{remark}[theorem]{Remark}

\numberwithin{equation}{section}

\newcommand{\norm}[1]{\lVert#1\rVert}

\newcommand{\cL}{\mathcal{L}}

\newcommand{\R}{\mathbb{R}}

\newcommand{\Z}{\mathbb{Z}}

\newcommand{\e}{\varepsilon}
\newcommand{\txt}[1]{\text{#1}}

\begin{document}

\title[Regularity theory of elliptic system in $\e$-scale flat domain]{Regularity theory of elliptic systems \\in $\e$-scale flat domains}


\author{Jinping Zhuge}
\address{Department of Mathematics, University of Chicago, Chicago, IL, 60637, USA.}
\curraddr{}
\email{jpzhuge@math.uchicago.edu}


\begin{abstract}
We consider the linear elliptic systems or equations in divergence form with periodically oscillating coefficients. We prove the large-scale boundary Lipschitz estimate for the weak solutions in domains satisfying the so-called $\e$-scale flatness condition, which could be arbitrarily rough below $\e$-scale. This particularly generalizes Kenig and Prange's work in \cite{KP15} and \cite{KP18} by a quantitative approach. Our result also provides a mathematical explanation on why the boundary regularity of the solutions of partial differential equations should be physically and experimentally expected even if the surfaces of mediums in real world may be arbitrarily rough at small scales. 
\end{abstract}

\maketitle
\section{Introduction}
\subsection{Motivation}
This paper is devoted to the boundary regularity of elliptic system/equation in a bounded domain whose boundary is arbitrarily rough at small scales. Precisely, let $D^\e$ be a bounded domain and $0\in \partial D^\e$. Consider the following linear elliptic system/equation
\begin{equation}\label{eq.main}
\left\{
\begin{aligned}
\nabla\cdot (A^\e \nabla u_\e) &= 0 \qquad &\txt{in } &D^\e_2, \\
u_\e  &=0  \qquad &\txt{on } & \Delta^\e_2,
\end{aligned}
\right.
\end{equation}
where $D^\e_t = D^\e\cap B_t(0)$, $\Delta^\e_t = \partial D^\e \cap B_t(0)$ and $A^\e(x) = A(x/\e)$. Throughout, we denote by $B_t(x)$ the ball centered at $x$ with radius $t$. As usual, in order to prove the uniform regularity for the solution $u_\e$, we must assume some self-similar structure for the coefficient matrix $A$, such as periodicity, almost-periodicity or randomness.

The pioneering work of the uniform regularity estimates in periodic homogenization dates back to the late 1980s by Avellaneda and Lin in a series of papers \cite{AL87,AL89,AL91}. In particular, by a compactness method, they proved that if $A$ is periodic and H\"{o}lder continuous, and the (non-oscillating) domain $D$ is $C^{1,\alpha}$, then
\begin{equation}\label{est.ptLip}
\norm{\nabla u_\e}_{L^\infty(D_1)} \le C\norm{\nabla u_\e}_{L^2(D_2)}.
\end{equation}
This estimate is called the uniform Lipschitz estimate since the constant $C$ is uniform in $\e>0$.
In the setting of almost-periodic or stochastic homogenization, the Lipschitz estimate for elliptic system in divergence form has been established in, e.g., \cite{S15,ASm16,AKM16,AGK16,ASh16,AM16,S17,Z17,SZ18}. On the other hand, for Neumann problem, the boundary Lipschitz estimate was first obtained by Kenig, Lin and Shen \cite{KLS13} with symmetric coefficients. The symmetry assumption was finally removed by Armstrong and Shen in \cite{ASh16}. Analogous results have been extended to a variety of equations, such as parabolic equations \cite{ABM18,GS15}, Stokes or elasticity systems \cite{GuS15,GX17,GZ19,GZ20}, higher-order equations \cite{NSX18,NX19}, nonlinear equations \cite{ASm14,ASm16,AFK19}, etc. We should point out that, in any work mentioned above, if we assume no smoothness on the coefficients, the pointwise Lipschitz estimate (\ref{est.ptLip}) should be replaced by the so called large-scale Lipschitz estimate
\begin{equation}\label{est.LSLip}
\bigg( \fint_{D_r} |\nabla u_\e|^2 \bigg)^{1/2} \le C\bigg( \fint_{D_2} |\nabla u_\e|^2 \bigg)^{1/2}, \qquad \text{for any } r\in (\e,2),
\end{equation}
where $C$ is independent of $r$ and $\e$. This estimate does not hold uniformly for $r\ll \e$ because, by a blow-up argument, the elliptic systems/equations may have unbounded $|\nabla u_\e|$. But (\ref{est.LSLip}) claims that $|\nabla u_\e|$ is bounded in the averaging sense above $\e$-scale. This phenomenon is physically natural as macroscopic (large-scale) smoothness of a solution for a PDE in an oscillating material should be expected if the material is well-structured microscopically.

Note that for all the aforementioned work, the domain is assumed to be smooth, namely, in $C^{1,\alpha}$ class. This class of domains is mathematically nearly sharp\footnote{The sharp class is $C^{1,\text{Dini}}$ domains \cite{Lie86}, yet $C^1$ domains are not sufficient.} for (\ref{est.ptLip}) even for Laplacian operator. However, it is not necessary for the estimate (\ref{est.LSLip}). Particularly, one may consider the boundaries that are rough or rapidly oscillating only at or below $\e$-scale, which appear naturally in reality. The study of PDEs with rapidly oscillating boundaries is currently an very active research area with various applications (particularly in fluid dynamics); see \cite{ABDG04,BG08,GD09,GM10,DG11,DP14} for example and reference therein. 
In terms of the uniform regularity in homogenization, some recent work has been done for the boundaries without any structure \cite{KP15,KP18,HP19}. In particular, Kenig and Prange \cite{KP18} proved the large-scale Lipschitz estimate when $D^\e$ is given by the graph of
\begin{equation}\label{est.bdry.KP}
x_d = \e \psi(x'/\e), \qquad \text{where } \psi \in W^{1,\infty}(\R^{d-1}).
\end{equation}
The originality of this result is that they made no assumption more than the Lipschitz regularity on the boundary (In their earlier work \cite{KP15}, $\psi$ is assumed to be in $C^{1,\alpha}$), which definitely bypasses the classical $C^{1,\alpha}$ assumption on the boundary. We point out that the Lipschitz estimate in \cite{KP15,KP18} was proved by the compactness method, following Avellaneda and Lin \cite{AL87}. Then more recently, Higaki and Prange \cite{HP19} used the similar idea to obtain the large-scale Lipschitz and $C^{1,\alpha}$ estimates for stationary Navier-Stokes equations over bumpy Lipschitz boundary given by (\ref{est.bdry.KP}). A nontrivial generalization of the boundary (\ref{est.bdry.KP}) has been studied by Gu and Zhuge for the system of nearly incompressible elasticity \cite{GZ20}, which particularly includes the graph given by
\begin{equation}\label{est.bdry.GZ}
x_d = \psi_0(x') + \e\psi_1(x'/\e), \quad \text{where } \psi_0\in C^{1,\alpha}(\R^{d-1}), \psi_1\in W^{1,\infty}(\R^{d-1}).
\end{equation}
This boundary could be viewed as a classical $C^{1,\alpha}$ graph with a small Lipschitz perturbation. Note that (\ref{est.bdry.GZ}) is still a Lipschitz graph.

The purpose of this paper is to generalize the large-scale Lipschitz estimate in rough domains without any regularity assumption (thus the boundary could be arbitrarily rough, including fractals and cusps), except for a quantitative large-scale flatness assumption defined below.
\begin{definition}\label{def.modulus}
	Let $D^\e$ be a bounded domain with $\e>0$. We say $D^\e$ is $\e$-scale flat with a modulus $\zeta:(0,1]\times (0,1]\mapsto [0,1]$, if for any $y\in \partial D^\e$ and $r\in (\e,1)$, there exists a unit (outward normal) vector $n_r = n_r(y) \in \R^d$ so that
	\begin{equation}\label{cond.flat}
	\begin{aligned}
	& B_r(y) \cap \{ x\in \R^d: (x-y)\cdot n_r < -r\zeta(r,\e/r) \} \\
	& \qquad \subset D^\e_r(y) \subset B_r(y) \cap \{ x\in \R^d: (x-y)\cdot n_r < r\zeta(r,\e/r) \},
	\end{aligned}
	\end{equation}
	where $D^\e_r(y) = D^\e \cap B_r(y)$.
\end{definition}

In other words, if $D^\e$ is $\e$-scale flat, $\Delta_r^\e$ is locally contained between two parallel hyperplanes whose distance is at most $2r \zeta(r,r/\e)$. Hence, (\ref{cond.flat}) may be viewed as a large-scale (since $r \ge \e$) quantitative Reifenberg flatness condition. The reason that we write $\zeta$ as a function of $r$ and $\e/r$, instead of $r$ and $\e$, may be seen in Definition \ref{def.esmooth} and the examples after. Clearly, Definition \ref{def.modulus} is a local property.

\begin{figure}[h]
	\begin{center}
		\includegraphics[scale =0.4]{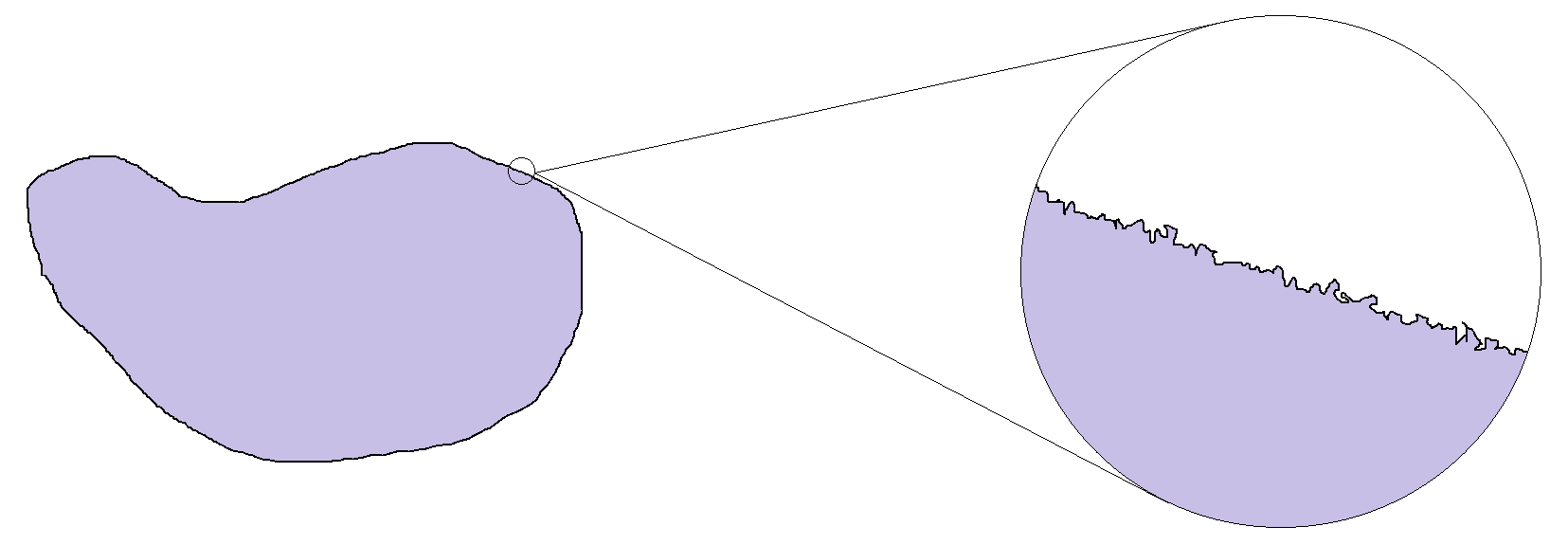}
	\end{center}
	\vspace{-1 em}
	\caption{A non-Lipschitz domain with arbitrary roughness at microscopic scales}\label{fig_1}
\end{figure}

Absolutely, the modulus $\zeta$ involved above will play a critical role in this paper and an additional quantitative condition is necessary for our purpose.

\begin{definition}\label{def.esmooth}
	Let $\eta:(0,1]\times (0,1] \to (0,1]$ be a continuous function. We say that $\eta$ is an ``admissible modulus'' if the following conditions hold:
	\begin{itemize}
		\item Flatness condition:
		\begin{equation}\label{cond.flatness}
		\lim_{t\to 0^+} \sup_{r,s\in (0,t)} \eta(r,s) = 0.
		\end{equation}

		\item A Dini-type condition:
		\begin{equation*}
		\lim_{t\to 0^+} \sup_{\e \in (0,t^2)} \int_{\e/t}^{t} \frac{\eta(r,\e/r)}{r} dr = 0.
		\end{equation*}
	\end{itemize}
	Moreover, we say $\eta$ is ``$\sigma$-admissible'' if $\eta^\sigma$ is an admissible modulus.
\end{definition}

We give three typical examples of $\sigma$-admissible moduli in the following.

\textbf{Example 1:} If $D^\e$ is uniformly $C^{1,\alpha}$, then $\zeta(r,s) = Cr^\alpha$. 

\textbf{Example 2:} If the boundary of $D^\e$ is locally given by the graph of $x_d = \e\psi(x'/\e)$ with $\psi \in C^0\cap L^\infty(\R^{d-1})$, then $\zeta(r,s) = Cs$. If $\psi(x')$ is uniformly $C^\beta$-H\"{o}lder continuous in $\R^{d-1}$ (not necessarily bounded), then $\zeta(r,s) = Cs^{1-\beta}$. Observe that either case here is a class much broader than (\ref{est.bdry.KP}).

\textbf{Example 3:} A typical microscopically oscillating boundary could be given by a graph $x_d = \psi_0(x') + \e\psi_1(x'/\e)$, where $\psi_0$ is a $C^{1,\alpha}$ function (capturing the macroscopic profile of the boundary) and $\psi_1$ satisfies either condition in Example 2 (capturing the microscopic details of the boundary). This is a combination of the previous two examples and $\zeta(r,s) = Cr^\alpha + C s^\beta$.

\subsection{Assumptions and main result}
We consider a family of oscillating elliptic operators in divergence form
\begin{equation}\label{eq.Le}
\nabla\cdot (A(x/\e) \nabla) :=  \frac{\partial}{\partial x_i} \bigg\{ a^{\alpha\beta}_{ij} \Big( \frac{x}{\e}\Big) \frac{\partial}{\partial x_j}\bigg\},
\end{equation}
with $1\le i,j\le d, 1\le \alpha,\beta \le m$ (Einstein's summation convention is used here and throughout), where $d$ represents the dimension and $m$ represents the number of equations. We assume that the coefficients matrix $A = (a^{\alpha\beta}_{ij})$ satisfies the following conditions:
\begin{itemize}
	\item Ellipticity: there exists $\Lambda>0$ such that
	\begin{equation*}
	\Lambda^{-1} |\xi|^2 \le a^{\alpha\beta}_{ij} \xi_i^\alpha \xi_j^\beta \le \Lambda |\xi|^2, \quad \txt{for any } \xi = (\xi_i^\alpha) \in \R^{m\times d}.
	\end{equation*}
	\item Periodicity:
	\begin{equation*}
	A(y+z) = A(y) \quad \text{for any }z \in \Z^d, y\in \R^d.
	\end{equation*}
\end{itemize}

Suppose $\{D_\e:\e>0\}$ is a family of bounded domains and $0\in \partial D_\e$. Let $B_t = B_t(0), D_t^\e = D^\e \cap B_t, \Delta_t^\e = \partial D_\e \cap B_t$ and $A^\e(x) = A(x/\e)$. We now give a definition for the weak solution of (\ref{eq.main}). We say $u_\e \in H^1(D_2^\e;\R^d)$ is a weak solution of (\ref{eq.main}) if for any $\phi\in C_0^\infty(D_2^\e)$,
\begin{equation*}
\int_{D_2^\e} A^\e \nabla u_\e \cdot \nabla \phi = 0
\end{equation*}
and $u_\e \varphi \in H_0^1(D_2^\e;\R^d)$ for any $\varphi \in C_0^\infty(B_2;\R^d)$.

Now, we state the main theorem of this paper.
\begin{theorem}\label{thm.main}
	Let $\e\in (0,1)$. Suppose $D^\e$ is $\e$-scale flat with a $\sigma$-admissible modulus $\zeta$ for some $\sigma\in (0,1/2)$. Let $u_\e$ be a weak solution of (\ref{eq.main}). Then for any $r\in (\e,1)$,
	\begin{equation}\label{est.Lip}
	\bigg( \fint_{D^\e_r} |\nabla u_\e|^2 \bigg)^{1/2} \le C\bigg( \fint_{D^\e_2} |\nabla u_\e|^2 \bigg)^{1/2},
	\end{equation}
	where $C$ depends only on $\Lambda, d, m$ and $\zeta$.
\end{theorem}

The above theorem justifies, from a mathematical point of view, a naturally expected phenomenon in real world: the macroscopic (large-scale) smoothness of the boundary implies the macroscopic (large-scale) smoothness of the solutions of PDEs, and the influence of the microscopic roughness of the boundary is physically invisible and experimentally undetectable. In other words, if the boundary of a domain is arbitrarily rough only at and below a certain small scale, then the solution (of a PDE) should be ``smooth'' near the boundary, in the averaging sense, above the same scale. 
Compared to Kenig and Prange's work \cite{KP15,KP18}, one of the main contributions of this paper is that we completely remove the Lipschitz regularity assumption on the boundary, which therefore could be arbitrarily rough at microscopic scales.

\begin{remark}
	The Reifenberg flat domains has been extensively studied in the past twenty years for the Calder\'{o}n-Zygmund estimates; see, e.g., \cite{BW04,BG08,BW08-2,BW10,MP12} and reference therein. The $\e$-scale flat domains defined above obviously are closely related and in some sense they could be viewed as a large-scale quantitative version of the Reifenberg flat domains. Because our domains have stronger flatness at large scales, instead of only Calder\'{o}n-Zgymund estimate, we have the Schauder estimate (\ref{est.Lip}).
\end{remark}

\begin{remark}
	To the best of our knowledge, Theorem \ref{thm.main} is the most general result (i.e., large-scale boundary Lipschitz estimate) so far in periodic homogenization for either oscillating or non-oscillating domains. Even for non-oscillating domains, Theorem \ref{thm.main} gives new results for $C^{1,\text{Dini}}$-type domains. More precisely, if the boundary of a domain is given by the graph $x_d = \psi(x')$ and the continuity modulus of $\nabla\psi$, denoted by $\eta(r)$, satisfies
	\begin{equation*}
	\int_{0}^{1} \frac{\eta(r)^\sigma}{r} dr < \infty,
	\end{equation*}
	for some $\sigma\in (0,1/2)$, then (\ref{est.Lip}) holds by Theorem \ref{thm.main}. This result is new since, as we have mentioned, previous results in homogenization all dealt with $C^{1,\alpha}$ domains with $\alpha>0$. The range of the exponent $\sigma$ may not be optimal, compared to the Lipschitz estimate for Laplacian; see, e.g., \cite{Lie86,MW12,AB18}. But this is an acceptable loss as we often encounter in homogenization theory.
\end{remark}


\begin{remark}
	In this final remark, we point out that, in scalar case (i.e., $m=1$), Theorem \ref{thm.main} may be strengthened to the domains with $\e$-scale convex points by the maximum principle. This will recover the Lipschitz estimate in convex domains for scalar elliptic equations. The precise definition and corresponding result are contained in Section 4.
\end{remark}

\subsection{Outline of the proof}
Now, we describe the key ideas of the proof of Theorem \ref{thm.main}. Instead of the compactness method, we will use a quantitative approach, the excess decay method, originating from \cite{ASm16}; see recent monographs \cite{AKM19,S18} for a comprehensive investigation. The key step in this approach is to establish an algebraic rate of convergence for the local homogenization problem, namely,
\begin{equation}\label{eq.Dre}
\left\{
\begin{aligned}
\nabla\cdot (A^\e \nabla u_\e) &= 0 \qquad &\txt{in } &D^\e_r, \\
u_\e  &=0  \qquad &\txt{on } & \Delta^\e_r,
\end{aligned}
\right.
\end{equation}
where $u_\e\in H^1(D^\e_r;\R^d)$ and $r\in (\e,1)$. However, since $\Delta_r^\e$ could be arbitrarily rough and has no self-similar structures, no a priori regularity estimate, such as the usual Meyers' estimate, is known for $\nabla u_\e$. This means that the rate of convergence may not be obtained obviously. Actually, the main challenge of this paper is to get rid of the arbitrary roughness of the boundary at small scales. To overcome this difficulty, we first approximate (\ref{eq.Dre}) by a local problem in a nicer domain $T_r^{\e,+} := B_r \cap \{ x\in \R^d: x\cdot n_r < r\zeta(r,\e/r) \}$. Note that $T_r^{\e,+} \supset D_r^\e$ is now a Lipschitz domain. Let $w_\e$ solves
\begin{equation}\label{eq.Tre}
\left\{
\begin{aligned}
\nabla\cdot (A^\e \nabla w_\e) &= 0 \qquad &\txt{in } &T^{\e,+}_r, \\
w_\e  &=u_\e  \qquad &\txt{on } & \partial T^{\e,+}_r,
\end{aligned}
\right.
\end{equation}
where $u_\e$ has been extended to the entire ball $B_r \supset T_r^{\e,+}$ with $u_\e \equiv 0$ in $B_r\setminus D_r^\e$. By the definition of ``admissible modulus'', $T_r^{\e,+}\setminus D_r^\e$ is contained in a slim layer. Hence, it is possible to estimate the error between $w_\e$ and $u_\e$ in terms of the ``admissible modulus''. Again, this estimate will depends on some a priori estimate of $\nabla u_\e$ which turns out to be an interesting byproduct of this paper. Actually, under condition (\ref{cond.flatness}) only, we are able to show the so-called large-scale Calder\'{o}n-Zygmund estimate (or the reverse H\"{o}lder inequality), namely
\begin{equation}\label{est.LSCZ.Br}
\bigg( \fint_{B_r} |M_t[\nabla u_\e]|^p \bigg)^{1/p} \le C\bigg( \fint_{B_{20r}} |M_t[\nabla u_\e]|^2 \bigg)^{1/2},
\end{equation}
for $t\in (\e/\e_0,\e_0)$, where
\begin{equation*}
M_t[F](x) = \bigg( \fint_{B_t(x)} |F|^{p_0} \bigg)^{1/{p_0}},
\end{equation*}
and $p_0 = 2d/(2+d)$. The averaging operator $M_t$ plays an essential role in getting rid of the boundary roughness at small scales and the estimate (\ref{est.LSCZ.Br}) is optimal in the sense that it does not hold uniformly for $t\ll \e$. The proof of (\ref{est.LSCZ.Br}) is roughly two steps. In the first step, we prove a large-scale Meyers' estimate with $p = 2+\delta$ and $\delta>0$ being tiny. This step has nothing to do with homogenization since it follows from the Gehring's inequality (a large-scale self-improvement property). In the second step, we take advantage of the uniform boundary Lipschitz estimate for $w_\e$ in periodic homogenization and a real-variable argument by Shen \cite[Chapter 3]{S18} to improve $p=2+\delta$ to any $p<\infty$. We should mention that a similar estimate in $C^{1,\alpha}$ domains was also obtained by Armstrong and Daniel \cite{AD16} (also see \cite[Chapter 7]{AKM19} and \cite[Corollary 4]{GNO20}). With (\ref{est.LSCZ.Br}) at our disposal, we can show that, for any $\sigma \in (0,1/2)$, $r\in (\e/\e_0,\e_0)$
\begin{equation}\label{est.Due-Dwe}
\bigg( \fint_{B_r} |\nabla u_\e - \nabla w_\e|^2 \bigg)^{1/2} \le C\big\{ \e/r + \zeta(r,\e/r) \big\}^\sigma \bigg( \fint_{B_{20r}} |\nabla u_\e|^2 \bigg)^{1/2}.
\end{equation}
This reduces the excess decay estimate of $u_\e$ into that of $w_\e$ with a controllable error.

The excess decay method to tackle the main theorem involves two critical quantities
\begin{equation}
\Phi(t)  =\frac{1}{t} \bigg( \fint_{D_t^\e } |u_\e|^2 \bigg)^{1/2}, \quad H(t) = \frac{1}{t}\inf_{q\in\R^d} \bigg( \fint_{D_{t}^\e } |u_\e - (n_t\cdot x)q|^2 \bigg)^{1/2}.
\end{equation}
Note that $\Phi(t)$ is almost equivalent to the $L^2$ average of $|\nabla u_\e|$ over $D_t^\e$, in view of the Poincar\'{e} and Caccioppoli inequalities. The structure of quantity $H(t)$ is novel and critical due to the lack of smoothness of the boundary. Recall that $n_t$ is defined in Definition \ref{def.modulus} which represents the approximate normal vector at $t$-scale. Thus $(n_t\cdot x)q$ is a directional linear function changing values only in the approximate normal direction. This kind of linear functions could approximate $w_\e$ well. Thus, using the rate of convergence for $w_\e$, the smoothness of the homogenized solution $w_0$ and (\ref{est.Due-Dwe}), we are able to show that there exists some constant $\theta\in (0,1/4)$ so that
\begin{equation}\label{est.EDS}
H(\theta r) \le \frac{1}{2}H(r) + C\big( \e/r+\zeta(r,\e/r)\big)^\sigma \Phi(20r),
\end{equation}
for any $r\in (\e/\e_0,\e_0)$. This estimate is called the excess decay estimate which leads to the main theorem by an iteration lemma (see Lemma \ref{lem.iteration}, which generalizes \cite[Lemma 8.5]{S17}) in which both conditions in Definition \ref{def.esmooth} are needed. Large part of the proof of (\ref{est.EDS}) nowadays has been rather standard in homogenization theory, although the special structure of $H(t)$ and the treatment of the discrepancy between the domains $D_r^\e$ and $T_r^{\e,+}$ give additional technical difficulties.

Finally, we shall emphasize that, throughout this paper, we will use $C$ and $c$ to denote constants that vary from line to line. Moreover, they depend at most on $d,m,p,\Lambda,\zeta$ and other non-scale parameters and never depend on $\e,r$ and $t$, etc.

\subsection{Organization of the paper}
The organization of the paper is as follows: In Section 2, we prove the large-scale Calder\'{o}n-Zygmund estimate and (\ref{est.Due-Dwe}). In Section 3, we prove Theorem \ref{thm.main}. In Section 4, we show the large-scale Lipschitz estimate in $\e$-scale convex domains for scalar elliptic equations.

\subsection*{Acknowledgement.} The author would like to thank professors Carlos Kenig and Zhongwei Shen for insightful discussions on the topic in this paper. The author also would like to thank the anonymous referee for careful reading on the manuscript and extremely helpful comments that significantly improve the quality of the paper.

\section{Calder\'{o}n-Zygmund Estimate}
This section is devoted to the large-scale Calder\'{o}n-Zygmund estimate and the approximation of the system (\ref{eq.main}) at mesoscopic scales.

\subsection{Large-scale self-improvement}
In this subsection, we assume $D^\e$ is an $\e$-scale flat domain with a modulus $\zeta$ satisfying (\ref{cond.flatness}) only. Let $u_\e\in H^1(D^\e_2;\R^d)$ be a weak solution of $\nabla\cdot (A^\e \nabla u_\e) = 0$ in $D^\e_2$ with vanishing Dirichlet boundary condition on $\Delta^\e_2$. Note that we may extend the $u_\e$ naturally to $B_2$ by
\begin{equation*}
\widetilde{u}_\e(x)= \left\{ 
\begin{aligned}
&u_\e(x) &\quad &\text{if } x\in D_2^\e\\
&0 &\quad &\text{if } B_2\setminus D_2^\e.
\end{aligned}
\right.
\end{equation*}
However, for convenience, we will still denote the extended function $\widetilde{u}_\e$ by $u_\e$. Note that $u_\e\in H^1(B_2;\R^d)$ and $\nabla u_\e = 0$ in $B_2\setminus D_2^\e$.

The following is the well-known Caccioppoli inequality.

\begin{lemma}[Caccioppoli inequality]\label{lem.Cacc.D}
	Let $t\in (0,1)$ and $B_{2t}(x) \subset B_2(0)$. Then
	\begin{equation}
	\bigg( \fint_{B_t(x)} |\nabla u_\e|^2 \bigg)^{1/2} \le  \frac{C}{t} \bigg( \fint_{B_{2t}(x)}|u_\e|^2 \bigg)^{1/2},
	\end{equation}
	where $C$ depends only on $d,m$ and $\Lambda$.
\end{lemma}

By Definition \ref{def.modulus} and (\ref{cond.flatness}), we may assume $r\zeta(r,\e/r)$ is a nondecreasing function and $\zeta(r,\e/r) \le 1/2$ for $r\in (\e,1)$ without loss of generality.

\begin{lemma}\label{lem.RHolder}
	Let $\e^*:= \e\zeta(\e,1)$ and $t\in (\e^*,1)$ and $B_{4t}(x) \subset B_2(0)$. Then 
	\begin{equation}\label{est.reverse.Bt}
	\bigg( \fint_{B_t(x)} |\nabla u_\e|^2 \bigg)^{1/2} \le C\bigg( \fint_{B_{4t}(x)}|\nabla u_\e|^{p_0} \bigg)^{1/{p_0}},
	\end{equation}
	where $1/{p_0} = 1/2+1/d$.
\end{lemma}
\begin{proof}
	We consider three cases separately.
	
	Case 1: $B_t(x) \subset B_2(0)\setminus D^\e$. This is trivial.
	
	Case 2: $B_{2t}(x) \subset D^\e$. The classical interior Caccioppoli inequality and the Sobolev-Poincar\'{e} inequality \cite[pp 164]{GT01} imply that
	\begin{equation*}
	\bigg( \fint_{B_t(x)} |\nabla u_\e|^2 \bigg)^{1/2} \le C\inf_{k \in \R} \frac{1}{t} \bigg( \fint_{B_{2t}(x)} |u_\e - k|^2 \bigg)^{1/2} \le C \bigg( \fint_{B_{2t}(x)} |\nabla u_\e|^{p_0} \bigg)^{1/{p_0}}.
	\end{equation*}
	
	Case 3: $B_{2t}(x) \cap \Delta_2 \neq \emptyset$. In this case, by our assumptions that $\zeta(r,\e/r)\le 1/2$ and $r\zeta(r,\e/r)$ is nondecreasing, $u_\e\equiv 0 $ on $B_{4t}(x) \setminus D^\e$ and $|B_{4t}(x) \setminus D^\e| \ge c|B_{4t}(x)|$ for $t\in (\e^*,1)$. Now, by Lemma \ref{lem.Cacc.D} and the Sobolev-Poincar\'{e} inequality \cite[pp 164]{GT01}, one arrives at
	\begin{equation*}
	\begin{aligned}
	\bigg( \fint_{B_t(x)} |\nabla u_\e|^2 \bigg)^{1/2} &\le C\bigg( \fint_{D^\e_{2t}(x)} |\nabla u_\e|^2 \bigg)^{1/2} \\
	&\le \frac{C}{t} \bigg( \fint_{B_{4t}(x)} | u_\e|^2 \bigg)^{1/2} \\
	&\le C\bigg( \fint_{B_{4t}(x)} | \nabla u_\e|^{p_0} \bigg)^{1/{p_0}},
	\end{aligned}
	\end{equation*}
	where $C$ depends only on $d,m$ and $\Lambda$.
\end{proof}

The above inequality is a reverse H\"{o}lder inequality. The problem is that it does not hold for all range of $t\in (0,1)$. If it does, the Gehring's inequality directly implies that $|\nabla u_\e| \in L^{2+\delta}(B_{r/2})$ for some $\delta >0$ (the usual Meyers' estimate). To overcome this difficulty, we will establish a ``large-scale self-improving property'' which is sufficient for our application.

Let $u_\e$ and $p_0$ be as above. For $t\in (\e^*,1)$, define
\begin{equation*}
M_t[F](x) = \bigg( \fint_{B_t(x)} |F|^{p_0} \bigg)^{1/{p_0}}.
\end{equation*}
\begin{lemma}\label{lem.RHolder.Mt}
	Fix $t\in (\e^*, 1)$. 
	
	(i) For any $s\in (0,t)$ and $B_s(x) \subset B_1(0)$,
	\begin{equation*}
	M_t[\nabla u_\e](x) \le C\bigg( \fint_{B_s(x)} |M_t[\nabla u_\e]|^{p_0} \bigg)^{1/{p_0}}.
	\end{equation*}
	
	(ii) For $s\in (t,1)$ and $B_{10s}(x) \subset B_{1}(0)$,
	\begin{equation*}
	\bigg( \fint_{B_s(x)} |M_t[\nabla u_\e]|^2 \bigg)^{1/2} \le  C\bigg( \fint_{B_{10s}(x)} |M_t[\nabla u_\e]|^{p_0} \bigg)^{1/{p_0}} 
	\end{equation*}
\end{lemma}
\begin{proof}
	(i) 
	This part has nothing to do with equations and may be proved for general $F$ instead of $\nabla u_\e $. By the definition and the Fubini's theorem
	\begin{equation*}
	\begin{aligned}
	\fint_{B_s(x)} |M_t[F](y)|^{p_0} dy & =  \fint_{B_s(x)} \fint_{B_t(y)} |F(z)|^{p_0} dz dy\\
	& = \fint_{B_s(x)} |B_t(x)|^{-1} \int_{B_{t+s}(x)} |F(z)|^{p_0} \chi_{\{z: |z-y|\le t \}} dz dy \\
	& = |B_t(x)|^{-1} \int_{B_{t+s}(x)} |F(z)|^{p_0} \fint_{B_s(x)}\chi_{\{y: |z-y|\le t \}} dy dz,
	\end{aligned}
	\end{equation*}
	where $\chi_E$ is the indicator function of the set $E$.
	Now, if $z\in B_t(x)$ and $s\in (0,t)$, then
	\begin{equation*}
	\fint_{B_s(x)}\chi_{\{y: |z-y|\le t \}} dy = \frac{|B_s(x) \cap B_t(z)|}{|B_s(x)|} \ge c > 0,
	\end{equation*}
	where $c$ is an absolute constant. This implies
	\begin{equation*}
	\fint_{B_s(x)} |M_t[F](y)|^{p_0} dy \ge c \fint_{B_t(x)} |F(y)|^{p_0} dy = c\Big( M_t[F](x) \Big)^{p_0}.
	\end{equation*}
	
	(ii) The second part is proved by Lemma \ref{lem.RHolder}. Using the H\"{o}lder inequality and the Fubini's theorem, we have
	\begin{equation*}
	\begin{aligned}	
	\bigg( \fint_{B_s(x)} |M_t[\nabla u_\e](y)|^2 dy \bigg)^{1/2} & = \bigg( \fint_{B_s(x)} \bigg( \fint_{B_t(y)} |\nabla u_\e(z)|^{p_0} dz\bigg)^{2/{p_0}} dy \bigg)^{1/2} \\
	& \le \bigg( \fint_{B_s(x)} \fint_{B_t(y)} |\nabla u_\e(z)|^2 dz dy \bigg)^{1/2} \\
	& = \bigg( \frac{1}{|B_t(x)| |B_s(x)|} \int_{B_{t+s}(x)} \int_{B_s(x)} |\nabla u_\e(z)|^2 \chi_{\{ y:|y-z| \le t \}} dy dz \bigg).
	\end{aligned}
	\end{equation*}
	Now, observe that
	\begin{equation*}
	\int_{B_s(x)} \chi_{\{ y:|y-z| \le t \}} dy \le |B_{t}(x)|.
	\end{equation*}
	It follows from the above estimates and our assumption $s>t$ that
	\begin{equation}\label{est.Mt2Du}
	\bigg( \fint_{B_s(x)} |M_t[\nabla u_\e](y)|^2 dy \bigg)^{1/2} \le C_d \bigg( \fint_{B_{2s}(x)} |\nabla u_\e|^2 \bigg)^{1/2} \le C \bigg( \fint_{B_{8s}(x)} |\nabla u_\e|^{p_0} \bigg)^{1/{p_0}}.
	\end{equation}
	It suffices to show
	\begin{equation}\label{est.Du.MtDu}
	\bigg( \fint_{B_{8s}(x)} |\nabla u_\e|^{p_0} \bigg)^{1/{p_0}} \le C\bigg( \fint_{B_{10s}(x)} |M_t[\nabla u_\e]|^{p_0} \bigg)^{1/{p_0}}. 
	\end{equation}
	The idea is the similar as part (i). The Fubini's theorem implies
	\begin{equation*}
	\begin{aligned}
	\fint_{B_{10s}(x)} |M_t[\nabla u_\e]|^{p_0} & = \fint_{B_{10s}(x)} \fint_{B_t(y)} |\nabla u_\e(z)|^{p_0} dz dy \\
	& = |B_t(x)|^{-1} \int_{B_{10s+t}(x)} |\nabla u_\e(z)|^{p_0} \fint_{B_{10s}(x)} \chi_{\{ y: |y-z|\le t \} } dy dz.
	\end{aligned}
	\end{equation*}
	Now if $z\in B_{8s}(x)$ and $t<s$
	\begin{equation*}
	\int_{B_{10s}(x)} \chi_{\{ y: |y-z|\le t \} } dy = |B_t(x)|.
	\end{equation*}
	This implies
	\begin{equation*}
	\fint_{B_{10s}(x)} |M_t[\nabla u_\e]|^{p_0} \ge c\fint_{B_{8s}(x)} |\nabla u_\e(z)|^{p_0} dz,
	\end{equation*}
	for some absolute constant $c>0$. The proof of (\ref{est.Du.MtDu}) then is complete.
\end{proof}

\begin{corollary}[Large-scale self-improvement]\label{coro.large-selfimp}
	There exists  $\delta>0$, depending only on $d,m$ and $\Lambda$, so that for any $t\in (\e^*,1)$ and $B_r(x)\subset B_1(0)$
	\begin{equation*}
	\bigg( \fint_{B_{r/2}(x)} |M_t[\nabla u_\e]|^{2+\delta} \bigg)^{1/{(2+\delta)}} \le C  \bigg( \fint_{B_{2r}(x)} |\nabla u_\e|^2 \bigg)^{1/2}.
	\end{equation*}
\end{corollary}

\begin{proof}
	Lemma \ref{lem.RHolder.Mt} shows that the function $M_t[\nabla u]$ satisfies the reverse H\"{o}lder inequality
	\begin{equation*}
	\bigg( \fint_{B_s(x)} |M_t[\nabla u_\e]|^2 \bigg)^{1/2} \le C \bigg( \fint_{B_{10s}(x)} |M_t[\nabla u_\e]|^{p_0} \bigg)^{1/{p_0}},
	\end{equation*}
	for all $x\in B_{2r}(0)$ and $s>0$ with $B_{10s}(x) \subset B_{2r}(0)$, where ${p_0} = 2d/(d+2) < 2$. Then, the standard Gehring's inequality (see \cite[Theorem 6.38]{GM12}) implies that there exists $\delta>0$ (depending only on the constant $C$ in the last inequality) so that
	\begin{equation*}
	\bigg( \fint_{B_{r/2}(x)} |M_t[\nabla u_\e]|^{2+\delta} \bigg)^{1/{(2+\delta)}} \le  C\bigg( \fint_{B_{r}(x)} |M_t[\nabla u_\e]|^{2} \bigg)^{1/{2}}.
	\end{equation*}
	This and the first inequality in (\ref{est.Mt2Du}) lead to the desired estimate.
\end{proof}

\subsection{Approximation}

Suppose $D^\e$ is an $\e$-scale flat domain with a modulus $\zeta$ satisfying (\ref{cond.flatness}). Let $u_\e$ be the weak solution (\ref{eq.main}) in $D^\e_2$. Recall that $u_\e$ may be extended naturally to the entire $B_2$ by zero-extension. For short, we denote 
\begin{equation*}
T_r^{\e,+}:= B_r \cap \{ x\in \R^d: x\cdot n < r\zeta(r,\e/r) \}
\end{equation*}
and 
\begin{equation*}
T_r^{\e,-}:= B_r \cap \{ x\in \R^d: x\cdot n < -r\zeta(r,\e/r) \}.
\end{equation*}

We construct an approximate solution of $u_\e$. Let $w_\e = w_\e^r$ (we will drop the superscript $r$ for simplicity, if there is no ambiguity) be the weak solution of
\begin{equation}\label{eq.we.D}
\left\{ 
\begin{aligned}
\nabla\cdot (A^\e \nabla  w_\e) &= 0 \qquad \text{in } T_r^{\e,+},\\
w_\e &= u_\e \qquad \text{on } \partial T_r^{\e,+}.
\end{aligned}
\right.
\end{equation}
Similar as $u_\e$, we extend $w_\e$ across the boundary by zero-extension.

\begin{lemma}\label{lem.approx}
	For every $r\in (\e,1)$
	\begin{equation}
	\bigg( \fint_{B_r} |\nabla u_\e - \nabla w_\e|^2 \bigg)^{1/2} \le C\zeta(r,\e/r)^\gamma \bigg( \fint_{B_{4r}} |\nabla u_\e|^2 \bigg)^{1/2},
	\end{equation}
	where $\gamma = 1/2- 1/(2+\delta)$.
\end{lemma}

\begin{proof}
	First of all, note that $w_\e - u_\e\in H^1_0(T_r^{\e,+};\R^d)$. Then,
	by testing $w_\e - u_\e$ on the system (\ref{eq.we.D}), we have
	\begin{equation}\label{eq.we.var}
	\int_{T_r^{\e,+}} A^\e\nabla w_\e \cdot \nabla(w_\e - u_\e) = 0.
	\end{equation}
	Let $\phi\in C_0^\infty(\R^d)$ is a smooth function so that $\phi = 1$ on $T_r^{\e,+}\setminus D^\e_r$. Then $(1-\phi)(w_\e-u_\e)\in H^1_0(D^\e_r;\R^d)$. Thus, since $u_\e$ is a weak solution in $D^\e_r$,
	\begin{equation}\label{eq.ue.var}
	\int_{D^\e_r} A^\e\nabla u_\e \cdot \nabla \big((1-\phi)(w_\e - u_\e) \big) = 0.
	\end{equation}
	Combining (\ref{eq.we.var}) and (\ref{eq.ue.var}), we have
	\begin{equation}\label{eq.Du-w}
	\int_{T_r^{\e,+}} A^\e\nabla (w_\e - u_\e)\cdot \nabla (w_\e - u_\e) = -\int_{D^\e_r} A^\e \nabla u_\e\cdot \nabla(\phi (w_\e - u_\e)).
	\end{equation}
	
	Now, we choose $\phi$ properly. Observe that $T_r^{\e,+}\setminus D^\e_r \subset T_r^{\e,+}\setminus T_r^{\e,-}$. In view of the definitions of $T_r^{\e,+}$ and $T_r^{\e,-}$, then we may choose $\phi$ so that $\phi = 1$ on $T_r^{\e,+}\setminus T_r^{\e,-}$ and $\phi = 0$ in $T_r^{\e,+}\cap \{ x\in \R^d: n_r\cdot x < -2r\zeta(r,\e/r) \}$. Moreover, $|\nabla \phi| \le C(r\zeta(r,\e/r))^{-1}$.
	
	Denote the set $T_r^{\e,+} \cap \{x\in \R^d: x\cdot n_r > -2r\zeta(r,\e/r)  \}$ by $P_r$. Note that $P_r$ is a lamina-like region whose radius is $r$ and thickness is $3r\zeta(r,\e/r)$. Thus $|P_r| \le Cr^d \zeta(r,\e/r)$. By the ellipticity condition, we have
	\begin{equation}\label{est.Dwe-Due.Tr}
	\Lambda^{-1} \int_{T_r^{\e,+}} |\nabla w_\e - \nabla u_\e|^2 \le \Lambda \bigg( \int_{P_r }|\nabla u_\e|^2\bigg)^{1/2} \bigg( \int_{P_r }|\nabla (\phi(w_\e - u_\e))|^2\bigg)^{1/2}.
	\end{equation}
	On one hand, the Poincar\'{e} inequality implies
	\begin{equation}\label{est.Dphi.weue}
	\bigg( \int_{P_r }|\nabla (\phi(w_\e - u_\e))|^2\bigg)^{1/2} \le C\bigg( \int_{T_r^{\e,+} }|\nabla (w_\e - u_\e)|^2\bigg)^{1/2}.
	\end{equation}
	On the other hand, we estimate
	\begin{equation*}
	J: = \bigg( \frac{1}{|D^\e_r|} \int_{ P_r } |\nabla u_\e|^2 \bigg)^{1/2}.
	\end{equation*}
	Let $t = r\zeta(r,\e/r)\ge \e^*$. By the Fubini's theorem, Lemma \ref{est.reverse.Bt} and Corollary \ref{coro.large-selfimp}, we have
	\begin{equation*}
	\begin{aligned}
	J & \le C\bigg( \frac{1}{|D^\e_r|} \int_{P_r} \fint_{B_{t/4}(x)} |\nabla u_\e(y)|^2dy dx  \bigg)^{1/2} \\
	& \le C \bigg( \frac{1}{|D^\e_r|} \int_{P_r} \bigg( \fint_{B_t(x)} |\nabla u_\e(y)|^{p_0} dy \bigg)^{2/{p_0}} dx  \bigg)^{1/2} \\
	& = C \bigg( \frac{1}{|D^\e_r|} \int_{P_r} |M_t[\nabla u_\e](x)|^2 dx  \bigg)^{1/2} \\
	& \le C\bigg( \frac{|P_r|}{|D^\e_r|} \bigg)^{1/2-1/(2+\delta)} \bigg( \fint_{B_r} |M_t[\nabla u_\e](x)|^{2+\delta} dx  \bigg)^{1/(2+\delta)} \\
	& \le C\zeta(r,\e/r)^{\gamma} \bigg( \fint_{D^\e_{4r}} |\nabla u_\e|^2 \bigg)^{1/2},
	\end{aligned}
	\end{equation*}
	where Corollary \ref{coro.large-selfimp} is used in the last inequality. 
	
	Inserting this into (\ref{est.Dwe-Due.Tr}) and using (\ref{est.Dphi.weue}), we obtain the desired estimate.
\end{proof}

\subsection{Large-scale Calder\'{o}n-Zygmund estimate}
Recall the assumption (\ref{cond.flatness}) on $\zeta$:
\begin{equation*}
\lim_{t\to 0^+} \sup_{r,s\in (0,t)} \zeta(r,s) = 0.
\end{equation*}
Thus, given any $\rho>0$, there exists $\e_0 >0$, depending only on the modulus $\rho, \zeta,\sigma,\Lambda,m$ and $d$, so that for any $r\in (\e/\e_0, \e_0)$ (so we need to assume $\e \le \e_0^2$),
\begin{equation}\label{est.zeta.rho}
\zeta(r,\e/r)^\gamma < \rho.
\end{equation}

Let $t>0$. Define the truncated maximal function in a ball $B$ by
\begin{equation}\label{key}
\mathcal{M}_{t,B}[F](x) = \sup \bigg\{ \bigg( \fint_{B_r(x)} |F|^2 \bigg)^{1/2}: r> t, B_r(x) \subset B \bigg\}.
\end{equation}
The following large-scale real-variable argument will be useful to us.
\begin{theorem}[A large-scale real-variable argument]\label{thm.real}
	Let $B_0$ be a ball in $\R^d$ and $F\in L^2(\alpha B_0)$. Let $q>2$ and $t>0$. Suppose that for each ball $B \subset 2B_0$ with $|B|\le c_0|B_0|$ and radius no less than $t$, there exist two measurable functions $F_B$ and $R_B$ on $2B$ such that $|F|\le |F_B| + |R_B|$ on $2B$, and
	\begin{equation}\label{est.RealCond}
	\begin{aligned}
	\bigg( \fint_{2B} |R_B|^q \bigg)^{1/q} & \le C_1 \bigg( \fint_{\alpha B} |F|^2 \bigg)^{1/2},\\
	\bigg( \fint_{2B} |F_B|^2 \bigg)^{1/2} &\le \eta \bigg( \fint_{\alpha B} |F|^2 \bigg)^{1/2},
	\end{aligned}
	\end{equation}
	where $C_1 >1,\alpha > 2$ and $0<c_0<1$. Then for any $2<p<q$ there exists $\eta_0>0$, depending only on $C_1,c_0,\alpha,p,q$, with the property that if $0\le \eta \le \eta_0$, then $F\in L^p(B_0)$ and
	\begin{equation}\label{est.Shen}
	\bigg( \fint_{B_0} |\mathcal{M}_{t,\alpha B_0}[F]|^p \bigg)^{1/p} \le C\bigg( \fint_{\alpha B_0} |F|^2 \bigg)^{1/2},
	\end{equation}
	where $C$ depends at most on $C_1,c_0,\alpha,p$ and $q$.
\end{theorem}

The above theorem is a simplified version of the result stated in \cite[Theorem 4.1 and Remark 4.2]{Shen20}, which is actually a corollary of the standard full-scale real-variable argument (e.g., \cite[Theorem 4.2.6]{S18}).

The following is the main result of this section which may be viewed as a large-scale Calder\'{o}n-Zygmund estimate.
\begin{theorem}\label{thm.LSSZ}
	Given any $p\in (2,\infty)$, there exists a constant $\e_0>0$, depending on $\zeta,p,\Lambda,m$ and $d$, so that for any $t\in (\e/\e_0,\e_0)$ and $r\in (0,\e_0)$, we have
	\begin{equation*}
	\bigg( \fint_{B_r} |M_t[\nabla u_\e]|^p \bigg)^{1/p} \le C\bigg( \fint_{B_{20r}} |M_t[\nabla u_\e]|^2 \bigg)^{1/2}.
	\end{equation*}
\end{theorem}
\begin{proof}
	This is a corollary of Theorem \ref{thm.real}. To see this, we need to approximate $u_\e$ at all scales no less than $t$. Let $\e_0 \ge r\ge t$ and $w_\e$ be the weak solution of (\ref{eq.we.D}). Since $t\in (\e/\e_0,\e_0)$, Lemma \ref{lem.approx} and (\ref{est.zeta.rho}) imply
	\begin{equation*}
	\bigg( \fint_{B_r} |\nabla u_\e - \nabla w_\e|^2 \bigg)^{1/2} \le C \rho \bigg( \fint_{B_{4r}} |\nabla u_\e|^2 \bigg)^{1/2},
	\end{equation*}
	where $\rho$ is arbitrary and $\e_0$ depends on the modulus $ \zeta, \rho, \sigma,\Lambda,m$ and $d$. Also note that the constant $C$ is independent of $\e_0$ or $\rho$. On the other hand, the boundary regularity of $w_\e$ and the energy estimate lead to
	\begin{equation*}
	\| \nabla w_\e \|_{L^\infty(B_{r/2})} \le C\bigg( \fint_{B_r} |\nabla u_\e|^2 \bigg)^{1/2}.
	\end{equation*}
	Even though the above two estimates were proved in $B_r = B_r(0)$ centered at the origin, it is obvious that they can be proved for any balls center at $y\in \Delta^\e_1$ or at $y\in D^\e_1$ with $B_{20r} \subset D^\e_2$ (interior estimate). By the zero-extension, these estimates actually hold in any balls with $B_{20r} \subset B_2$. We now choose $\rho$ sufficiently small (so that $\e_0$ is small and fixed) and apply Theorem \ref{thm.real} with $q= \infty, \alpha = 8, F = |\nabla u_\e|, F_{B} = |\nabla u_\e - \nabla w_r|$ and $R_B = |\nabla w_r|$. It follows that
	\begin{equation}\label{est.1126}
	\bigg( \fint_{B_r} |\mathcal{M}_{t,B_{8r}}[\nabla u_\e]|^p \bigg)^{1/p} \le C \bigg( \fint_{B_{8r}} |\nabla u_\e|^2 \bigg)^{1/2}.
	\end{equation}
	Observe that the truncated maximal function $\mathcal{M}_{t,B_{8r}}[\nabla u_\e]$ may be replaced by the averaging operator $M_t[\nabla u_\e]$ since the former is larger. Meanwhile, (\ref{est.Mt2Du}) and (\ref{est.Du.MtDu}) show that $|\nabla u_\e|$ on the right-hand side of (\ref{est.1126}) can also be replaced by $M_t[\nabla u_\e]$ by enlarging the ball. This proves the desired estimate.
\end{proof}	


\subsection{Improved approximation}
In the following theorem, we use Theorem \ref{thm.LSSZ} to improve the estimate in Lemma \ref{lem.approx}. Precisely, we improve the small exponent $\gamma = 1/2-1/(2+\delta)$ to any $\sigma\in (0,1/2)$.
\begin{theorem}\label{thm.improved.approx}
	For any $\sigma \in (0,1/2)$, there exist $\e_0>0$ and $C>0$, depending only on $\sigma, \zeta,p,\Lambda,m$ and $d$, so that for any $r\in (\e/\e_0,\e_0)$,
	\begin{equation}
	\bigg( \fint_{B_r} |\nabla u_\e - \nabla w_\e|^2 \bigg)^{1/2} \le C\big\{ \e/r + \zeta(r,\e/r) \big\}^\sigma \bigg( \fint_{B_{20r}} |\nabla u_\e|^2 \bigg)^{1/2},
	\end{equation}
	where $w_\e = w_\e^r$ is the weak solution of (\ref{eq.we.D}).
\end{theorem}
\begin{proof}
	Given $\sigma\in (0,1/2)$, let $p$ be given by $\sigma = 1/2-1/p$. With this $p$, let $\e_0$ be the constant given in Theorem \ref{thm.LSSZ}. As in the proof of Lemma \ref{lem.approx},
	it suffices to improve the estimate of $J^*$ in the proof of Lemma \ref{lem.approx}, where
	\begin{equation*}
	J^*: = \bigg( \frac{1}{|D^\e_r|} \int_{ P_r^* } |\nabla u_\e|^2 \bigg)^{1/2},
	\end{equation*}
	where $P_r^* = T_r^{\e,+} \cap \{x\in \R^d:x\cdot n_r > -2t  \}$ and $t=\max\{ \e/\e_0,r\zeta(r,\e/r)\} \in (\e/\e_0,\e_0)$. Note that $t\le r$ for $r\in (\e/\e_0,\e_0)$. Now, the Fubini's theorem and Lemma \ref{est.reverse.Bt} imply
	\begin{equation*}
	\begin{aligned}
	\int_{ P_r^* } |\nabla u_\e|^2 &\le C \int_{P_r^*} \fint_{B_{t/4}(x)} |\nabla u_\e(y)|^2dy dx \\
	& \le  C\int_{P_r^*} \bigg( \fint_{B_{t}(x)} |\nabla u_\e(y)|^{p_0} dy \bigg)^{2/{p_0}} dx \\
	& \le C\int_{P_r^*} |M_{t}[\nabla u_\e](x)|^2 dx.
	\end{aligned}
	\end{equation*}
	It follows that
	\begin{equation*}
	\begin{aligned}
	J^*
	& \le C \bigg( \frac{|P_r^*|}{|D^\e_r|} \bigg)^{1/2-1/p} \bigg( \fint_{B_r} |M_{t}[\nabla u_\e](x)|^{p} dx  \bigg)^{1/p} \\
	& \le C(t/r)^{\sigma} \bigg( \fint_{D^\e_{20r}} |\nabla u_\e|^2 \bigg)^{1/2}\\
	& \le C\big\{ \e/r + \zeta(r,\e/r) \big\}^\sigma \bigg( \fint_{B_{20r}} |\nabla u_\e|^2 \bigg)^{1/2}, 
	\end{aligned}
	\end{equation*}
	where we have used Theorem \ref{thm.LSSZ} and the fact $|P_r^*| \simeq tr^{d-1}$.
	This implies the desired result by an argument as in the proof of Lemma \ref{lem.approx}.
\end{proof}

\section{Lipschitz Estimate}
\subsection{Boundary geometry}
Let $D^\e$ be a bounded $\e$-scale flat domain with $0\in \partial D^\e$. As usual, define $D^\e_t = D^\e\cap B_t(0)$ and $\Delta^\e_t = \partial D^\e\cap B_t(0)$. By Definition \ref{def.esmooth}, for any $t\in (\e,1)$, there exists a unit ``outer normal'' vector $n_t\in \mathbb{S}^{d-1}$ such that
\begin{equation*}
\begin{aligned}
T_t^{\e,-} &:= \{x\in \R^d: x\cdot n_t < - t \zeta(t,\e/t) \} \cap B_t(0)\\
& \subset D^\e_t \subset T_t^{\e,+}:= \{x\in \R^d: x\cdot n_t < t\zeta(t,\e/t) \}\cap B_t(0),
\end{aligned}
\end{equation*}
where $\zeta(t,s)$ is a $\sigma$-admissible  modulus. This particularly implies that both $T_t^{\e,-}$ and $T_t^{\e,+}$ approximate $D^\e_t$ well at almost all scales with $\e\ll t \ll 1$. Moreover, 
\begin{equation*}
|T_t^{\e,+}\setminus T_t^{\e,-}| \le C_d t^{d} \zeta(t,\e/t) = C_d \zeta(t,\e/t) |D^\e_t|.
\end{equation*}

The outer normal $n_t$ of the flat boundary of $T_t^{\pm}$ will play an important role in our proof. Intuitively, it represents a  macroscopically approximate direction perpendicular to the boundary near $0$ at $t$-scale and coincides with the usual outer normal if the boundary is smooth (say, $C^{1,\alpha}$). The following lemma shows that $n_t$ changes gently with $t\in (\e,1)$.

\begin{lemma}\label{lem.n_r}
	Let $\e\le s\le r\le 1$, then
	\begin{equation*}
	|n_r - n_s| \le C\frac{r\zeta(r,\e/r)}{s}.
	\end{equation*}
\end{lemma}
Let us sketch the proof of Lemma \ref{lem.n_r}. By the definition of $\e$-flat domain, the set $T_t^{\e,+}\setminus T_t^{\e,-}$ (a thin cylinder with height $2t\zeta(t,\e/t)$) contains the boundary $\Delta_t^\e$. If $s\le  r$, then $\Delta_s^\e \subset \Delta_r^\e$. This implies that the cylinder $T_s^{\e,+}\setminus T_s^{\e,-}$ is almost contained in $T_r^{\e,+}\setminus T_r^{\e,-}$. Let $\theta$ be the angle between $n_s$ and $n_r$, then a geometric observation shows that
\begin{equation*}
|n_r - n_s| \approx \sin \theta \le \frac{r\zeta(r,\e/r)}{s}.
\end{equation*}
This is the desired estimate.

\subsection{Excess quantities}
Let $u_\e \in H^1(D^\e_2;\R^d)$ be the weak solution of
\begin{equation}\label{eq.ue.bdry}
\left\{
\begin{aligned}
\nabla\cdot (A^\e \nabla u_\e) &= 0 \qquad &\txt{in } &D^\e_{2}, \\
u_\e  &=0  \qquad &\txt{on } & \Delta^\e_{2}.
\end{aligned}
\right.
\end{equation}
We define two quantities $\Phi$ and $H$ as follows: for any $v\in H^1(D^\e_t;\R^d)$
\begin{equation}\label{def.Phi}
\Phi(t;v)  =\frac{1}{t} \bigg( \fint_{D^\e_t } |v|^2 \bigg)^{1/2} 
\end{equation}
and
\begin{equation}\label{def.H}
H(t;v) = \frac{1}{t}\inf_{q\in\R^d} \bigg( \fint_{D^\e_{t} } |v - (n_t\cdot x)q|^2 \bigg)^{1/2}.
\end{equation}
Put $\Phi(t) = \Phi(t;u_\e), H(t) = H(t;u_\e)$ for short. In view of the Poincar\'{e} and Caccioppoli inequalities, the large-scale Lipschitz estimate is equivalent to the estimate of $\Phi(t)$ for $t\in (\e,1)$.

We need some basic properties of $\Phi$ and $H$. First of all, it is clear that for any $r>0$
\begin{equation}\label{eq.H<Phi}
H(r) \le \Phi(r). 
\end{equation}
Using the flatness of $D^\e$, it is not hard to see
\begin{equation}\label{est.HPhi}
\sup_{r\le s\le 2r}\Phi(s) \le C\Phi(2r),
\end{equation}
for any $r\in (\e,1)$.

\begin{lemma}\label{lem.h.bdry}
	There exists a function $h:(0,2) \mapsto [0,\infty)$ so that for any $r\in (\e,1)$
	\begin{equation}\label{est.hHPhi}
	\left\{
	\begin{aligned}
	h(r) & \le C(H(r) + \Phi(r)) \\
	\Phi(r) &\le H(r) + h(r) \\
	\sup_{r\le s,t\le 2r} |h(s) - h(t)| &\le CH(2r) + C\zeta(2r,\e/2r)\Phi(2r).
	\end{aligned}\right.
	\end{equation}
\end{lemma}
\begin{proof} 
	Let $q_r$ be the vector that minimizes $H(r)$, namely,
	\begin{equation}\label{def.re.H}
	H(r) = \frac{1}{r} \bigg( \fint_{D^\e_{r} } |u_\e - (n_r\cdot x)q_r|^2 \bigg)^{1/2}.
	\end{equation}
	Define $h(r) = |q_r|$. To see the first inequality of (\ref{est.hHPhi}), observe that
	\begin{equation*}
	\bigg( \fint_{D_r^\e} |n_r\cdot x|^2 \bigg)^{1/2} \ge cr,
	\end{equation*}
	for some absolute constant $c>0$. Hence,
	\begin{equation*}
	\begin{aligned}
	h(r) & \le |q_r| \frac{C}{r}\bigg( \fint_{D_r^\e} |n_r\cdot x|^2 \bigg)^{1/2} = \frac{C}{r}\bigg( \fint_{D_r^\e} |(n_r\cdot x)q_r|^2 \bigg)^{1/2} \\
	& \le \frac{C}{r} \bigg( \fint_{D^\e_{r} } |u_\e - (n_r\cdot x)q_r|^2 \bigg)^{1/2} + \frac{C}{r} \bigg( \fint_{D^\e_{r} } |u_\e|^2 \bigg)^{1/2}.
	\end{aligned}
	\end{equation*}
	This proves the first inequality of (\ref{est.hHPhi}). The second inequality follows easily by the triangle inequality.
	
	Now, we prove the third inequality of (\ref{est.hHPhi}). By the first inequality of (\ref{est.hHPhi}) and (\ref{eq.H<Phi}), $h(r) \le C\Phi(r)$. Since $r\in (\e,1)$, the flatness condition implies that $|D^\e_r|\simeq r^d$ and $|n_{2r}\cdot x| \ge C>0$ in a subset of $D^\e_r$ with volume comparable to $r^d$.
	Therefore, for any $s,t\in [r,2r]$, one has
	\begin{equation}\label{est.qsqt}
	\begin{aligned}
	|q_s - q_t| &\le \frac{C}{r} \bigg( \fint_{D^\e_{r} } |(n_{2r}\cdot x)(q_s - q_t)|^2 \bigg)^{1/2} \\
	& \le \frac{C}{r} \bigg( \fint_{D^\e_{r} } |u_\e - (n_{2r}\cdot x)q_s)|^2 \bigg)^{1/2} + \frac{C}{r} \bigg( \fint_{D^\e_{r} } |u_\e - (n_{2r}\cdot x)q_t)|^2 \bigg)^{1/2}.
	\end{aligned}
	\end{equation}
	We estimate the first term. Using Lemma \ref{lem.n_r} and (\ref{est.HPhi}), we have
	\begin{equation*}
	\begin{aligned}
	& \frac{1}{r} \bigg( \fint_{D^\e_{r} } |u_\e - (n_{2r}\cdot x)q_s)|^2 \bigg)^{1/2} \\
	&\le \frac{C}{r} \bigg( \fint_{D^\e_{s} } |u_\e - (n_{s}\cdot x)q_s)|^2 \bigg)^{1/2} + |n_{2r}-n_s||q_s| \\
	& \le \frac{C}{s} \inf_{q\in \R^d}\bigg( \fint_{D^\e_{s} } |u_\e - (n_{s}\cdot x)q)|^2 \bigg)^{1/2} + C\zeta(2r,\e/2r) \Phi(2r) \\
	& \le \frac{C}{s}  \bigg( \fint_{D^\e_{s} } |u_\e - (n_{s}\cdot x)q_{2r})|^2 \bigg)^{1/2} + C\zeta(2r,\e/2r) \Phi(2r) \\
	& \le \frac{C}{2r} \bigg( \fint_{D^\e_{2r} } |u_\e - (n_{2r}\cdot x)q_{2r})|^2 \bigg)^{1/2} + |n_{2r} - n_s| |q_{2r}| + C\zeta(2r,\e/2r) \Phi(2r) \\
	& \le CH(2r) + C\zeta(2r,\e/2r) \Phi(2r),
	\end{aligned}
	\end{equation*}
	The estimate for the second term of (\ref{est.qsqt}) is similar. Hence, for any $s,t\in [r,2r]$,
	\begin{equation*}
	|h(s) - h(t)| \le |q_s -q_t| \le CH(2r) + C\zeta(2r,\e/2r) \Phi(2r),
	\end{equation*}
	as desired.
\end{proof}

\subsection{Excess decay estimate}
In this subsection, we will prove a convergence rate for the system (\ref{eq.we.D}). Since $u_\e\in H^1(B_{2r};\R^d)$ for any $r\in (0,1)$, by the co-area formula, without loss of generality, we may assume $u_\e|_{\partial T_r^{\e,+}} \in H^1(\partial T_r^{\e,+};\R^d)$. Moreover, note that $T_r^{\e,+}$ is a Lipschitz domain. Let $w_\e = w_\e^r$ be the weak solution of (\ref{eq.we.D}). By a standard result of the convergence rate in periodic homogenization \cite{S17}, we have
\begin{equation}\label{est.rate}
\norm{w_\e - w_0}_{L^2(T_r^{\e,+})} \le C r (\e/r)^{1/2} \norm{\nabla u_\e}_{L^2(D^\e_{2r})},
\end{equation}
where $w_0 = w_0^r$ is the solution of the homogenized system in $T_r^{\e,+}$:
\begin{equation}\label{eq.w0}
\left\{ 
\begin{aligned}
\nabla\cdot (\widehat{A} \nabla w_0) &= 0 \qquad \text{in } T_r^{\e,+},\\
w_0 &= u_\e \qquad \text{on } \partial T_r^{\e,+},
\end{aligned}
\right.
\end{equation}
where $\widehat{A}$ is the homogenized coefficient matrix of $A$.

Let
\begin{equation*}
\widetilde{H}(r,a;w_\e) = \inf_{q\in \R^d} \frac{1}{a r}\bigg( \fint_{T_r^{\e,+}\cap B_{a r}} |w_\e -(n_r\cdot x) q|^2 \bigg)^{1/2}.
\end{equation*}
By using the smoothness of $w_0$ near the flat boundary, we may show
\begin{lemma}\label{lem.we.decay}
	There is a constant $\theta\in (0,1)$ such that for any $r\in (\e,1)$
	\begin{equation*}
	\widetilde{H}(r,\theta;w_\e)
	\le \frac{1}{2} \widetilde{H}(r,1;w_\e) + C \big( (\e/r)^{1/2} + \zeta(r,\e/r) \big) \Phi(4r).
	\end{equation*}
\end{lemma}

\begin{proof}
	By the $C^{1,1}$ regularity of $w_0$ on the flat boundary, we know
	\begin{equation*}
	\norm{\nabla^2 w_0}_{L^\infty(T_r^{\e,+}\cap B_{r/4})} + r^{-1} \norm{\nabla w_0}_{L^\infty(T_r^{\e,+}\cap B_{r/4})} \le Cr^{-1} \bigg( \fint_{T_r^{\e,+}\cap B_{r/2}} |\nabla w_0|^2 \bigg)^{1/2}.
	\end{equation*}
	Let $x_r$ be the point on the flat boundary $\partial T_r^{\e,+}\cap B_r$ so that it is the closest to the origin. By our assumption, $|x_r| \le Cr\zeta(r,\e/r)$. Clearly, since $w_0$ is identically $0$ on the flat boundary, the tangential derivatives vanish at $x_r$, i.e.,
	\begin{equation*}
	(I -n_r\otimes n_r)\nabla w_0(x_r) = 0.
	\end{equation*}
	Hence,
	\begin{equation*}
	\nabla w_0(x_r) = (n_r\otimes n_r) \nabla w_0(x_r) = n_{r}(n_r\cdot \nabla w_0(x_r)).
	\end{equation*}
	Consequently, by the Taylor expansion of $w_0$ at $x_r$
	\begin{equation*}
	|w_0(x) - w_0(x_r) - (x-x_r)\cdot \nabla w_0(x_r)| \le C|x-x_r|^2 \norm{\nabla^2 w_0}_{L^\infty(T_r^{\e,+}\cap B_{r/4})},
	\end{equation*}
	we have that for any $x\in T_r^{\e,+}\cap B_{r/4}$
	\begin{equation*}
	\begin{aligned}
	& |w_0(x) - (x\cdot n_r) (n_r \cdot \nabla w_0(x_r))| 
	\\ & \quad \le C|x-x_r|^2 \norm{\nabla^2 w_0}_{L^\infty(T_r^{\e,+}\cap B_{r/4})} + C|x_r|\bigg( \fint_{T_r^{\e,+}\cap B_{r/2}} |\nabla w_0|^2 \bigg)^{1/2}.
	\end{aligned}
	\end{equation*}
	Therefore, for any $\theta\in (0,r/4)$,
	\begin{equation}\label{est.w0.Btr}
	\begin{aligned}
	&\inf_{q\in \R^d} \frac{1}{\theta r} \bigg( \fint_{T_r^{\e,+} \cap B_{\theta r}} |w_0 - (n_r\cdot x) q|^2\bigg)^{1/2} \\
	& \qquad \le C\theta r \norm{\nabla^2 w_0}_{L^\infty(T_r^{\e,+}\cap B_{r/4})} + C_\theta  \zeta(r,\e/r) \bigg( \fint_{T_r^{\e,+}\cap B_{r/2}} |\nabla w_0|^2 \bigg)^{1/2}.
	\end{aligned}
	\end{equation}
	
	On the other hand, observe that, for any $q\in \R^d$, $w_0 - (n_r\cdot (x-x_r)) q$ is also a weak solution, i.e.,
	\begin{equation}\label{eq.w0Qq}
	\left\{ 
	\begin{aligned}
	\cL_0 (w_0 - (n_r\cdot (x-x_r)) q) &= 0 \qquad \text{in } T_r^{\e,+},\\
	w_0 - (n_r\cdot (x-x_r)) q &= 0 \qquad \text{on } \partial T_r^{\e,+}\cap B_r.
	\end{aligned}
	\right.
	\end{equation}
	Then, the $C^{1,1}$ estimate for $w_0 - (n_r\cdot (x-x_r)) q$ and the boundary Caccioppoli inequality give
	\begin{equation}\label{est.rD2we}
	r \norm{\nabla^2 w_0}_{L^\infty(T_r^{\e,+}\cap B_{r/4})}
	\le C \inf_{q\in \R^d} \frac{1}{r} \bigg( \fint_{T_r^{\e,+} \cap B_{ 3r/4}} |w_0 - (n_r\cdot (x-x_r)) q|^2\bigg)^{1/2}.
	\end{equation}
	To estimate the right-hand side of the last inequality, we let $q_1$ be the vector so that
	\begin{equation}\label{est.for.q_1}
	\inf_{q\in \R^d} \frac{1}{r} \bigg( \fint_{T_r^{\e,+} \cap B_{ 3r/4}} |w_0 - (n_r\cdot x) q|^2\bigg)^{1/2} =  \frac{1}{r} \bigg( \fint_{T_r^{\e,+} \cap B_{ 3r/4}} |w_0 - (n_r\cdot x) q_1|^2\bigg)^{1/2}.
	\end{equation}
	From a simple geometrical observation, $|n_r\cdot x| \ge Cr$ in a large portion of $T_r^{\e,+} \cap B_{ 3r/4}$. Hence, by the triangle inequality, one has
	\begin{equation*}
	\frac{1}{r} \bigg( \fint_{T_r^{\e,+} \cap B_{ 3r/4}} |w_0 - (n_r\cdot x) q_1|^2\bigg)^{1/2} \ge c |q_1| - \frac{1}{r} \bigg( \fint_{T_r^{\e,+}\cap B_{3r/4} } |w_0|^2 \bigg)^{1/2}.
	\end{equation*}
	Clearly,
	\begin{equation*}
		\inf_{q\in \R^d} \frac{1}{r} \bigg( \fint_{T_r^{\e,+} \cap B_{ 3r/4}} |w_0 - (n_r\cdot x) q|^2\bigg)^{1/2} \le \frac{1}{r}\bigg( \fint_{T_r^{\e,+}\cap B_{3r/4} } |w_0|^2 \bigg)^{1/2}.
	\end{equation*}
	It follows from the last three inequalities that
	\begin{equation}\label{est.q1}
	|q_1| \le \frac{C}{r} \bigg( \fint_{T_r^{\e,+}\cap B_{3r/4} } |w_0|^2 \bigg)^{1/2}.
	\end{equation}
	As a consequence, the triangle inequality yields
	\begin{equation}\label{key}
	\begin{aligned}
	& \inf_{q\in \R^d} \frac{1}{r} \bigg( \fint_{T_r^{\e,+} \cap B_{ 3r/4}} |w_0 - (n_r\cdot (x-x_r)) q|^2\bigg)^{1/2} -  \frac{1}{r} \bigg( \fint_{T_r^{\e,+} \cap B_{ 3r/4}} |w_0 - (n_r\cdot x) q_1|^2\bigg)^{1/2} \\
	& \le \inf_{q\in \R^d} \frac{1}{r} \bigg( \fint_{T_r^{\e,+} \cap B_{ 3r/4}} |(n_r\cdot x) (q-q_1) - (n_r\cdot x_r)q|^2\bigg)^{1/2} \\
	& \le \inf_{q\in \R^d} \frac{1}{r} \bigg( \fint_{T_r^{\e,+} \cap B_{ 3r/4}} | (n_r\cdot x_r)q_1|^2\bigg)^{1/2} \\
	& \le C \frac{|x_r|}{r^2} \bigg( \fint_{T_r^{\e,+}\cap B_{3r/4} } |w_0|^2 \bigg)^{1/2},
	\end{aligned}
	\end{equation}
	where we have used (\ref{est.q1}) in the last inequality. This shows that
	\begin{equation}\label{key}
	\begin{aligned}
	& \inf_{q\in \R^d} \frac{1}{r} \bigg( \fint_{T_r^{\e,+} \cap B_{ 3r/4}} |w_0 - (n_r\cdot (x-x_r)) q|^2\bigg)^{1/2} \\
	& \quad \le \inf_{q\in \R^d} \frac{1}{r} \bigg( \fint_{T_r^{\e,+} \cap B_{ 3r/4}} |w_0 - (n_r\cdot x) q|^2\bigg)^{1/2} + C \frac{\zeta(r,\e/r)}{r} \bigg( \fint_{T_r^{\e,+}\cap B_{3r/4} } |w_0|^2 \bigg)^{1/2}.
	\end{aligned}
	\end{equation}
	Now, by (\ref{est.rD2we}) and the last inequality, we obtain
	\begin{equation}\label{est.Dw0.Br}
	\begin{aligned}
	r \norm{\nabla^2 w_0}_{L^\infty(T_r^{\e,+}\cap B_{r/4})} & \le C\inf_{q\in \R^d} \frac{1}{r} \bigg( \fint_{T_r^{\e,+} } |w_0 - (n_r\cdot x) q|^2\bigg)^{1/2} \\
	& \qquad + C\frac{\zeta(r,\e/r)}{r} \bigg( \fint_{T_r^{\e,+}\cap B_{3r/4} } |w_0|^2 \bigg)^{1/2}.
	\end{aligned}
	\end{equation}
	Inserting this into (\ref{est.w0.Btr}) and using the Poincar\'{e} inequality, we arrive at
	\begin{equation*}
	\begin{aligned}
	& \inf_{q\in \R^d} \frac{1}{\theta r} \bigg( \fint_{T_r^{\e,+} \cap B_{\theta r}} |w_0 - (n_r\cdot x) q|^2\bigg)^{1/2} \\
	&\le C\theta \inf_{q\in \R^d} \frac{1}{r} \bigg( \fint_{T_r^{\e,+} } |w_0 - (n_r\cdot x) q|^2\bigg)^{1/2} + C_\theta \zeta(r,\e/r) \bigg( \fint_{T_r^{\e,+}\cap B_{3r/4} } |\nabla w_0|^2 \bigg)^{1/2}.
	\end{aligned}
	\end{equation*}
	Using (\ref{est.rate}), we have
	\begin{equation*}
	\begin{aligned}
	&\inf_{q\in \R}\frac{1}{\theta r} \bigg( \fint_{T_r^{\e,+}\cap B_{\theta r}} |w_\e -(n_r\cdot x) q|^2 \bigg)^{1/2} \\
	&\le C\theta \inf_{q\in \R} \frac{1}{r} \bigg( \fint_{T_r^{\e,+}} |w_\e -(n_r\cdot x) q|^2 \bigg)^{1/2} + C_\theta \big( (\e/r)^{1/2} + \zeta(r,\e/r) \big) \bigg( \fint_{D^\e_{2r}} |\nabla u_\e|^2 \bigg)^{1/2},
	\end{aligned}
	\end{equation*}
	where we also used the energy estimates
	\begin{equation*}
	\bigg( \fint_{T_r^{\e,+}} |\nabla w_0|^2 \bigg)^{1/2} \le C\bigg( \fint_{T_r^{\e,+}} |\nabla w_\e|^2 \bigg)^{1/2} \le C\bigg( \fint_{T_r^{\e,+}} |\nabla u_\e|^2 \bigg)^{1/2}.
	\end{equation*}
	Finally, choosing $\theta$ so that $C\theta = \frac{1}{2}$ and applying the Caccioppoli inequality to $u_\e$, we obtain the desired estimate.
\end{proof}


Combining Theorem \ref{thm.improved.approx} and Lemma \ref{lem.we.decay}, we have
\begin{lemma}\label{lem.H}
	Let $\sigma\in (0,1/2)$. There are $\theta\in (0,1)$ and $\e_0\in (0,1)$ such that if $r\in (\e/\e_0,\e_0)$
	\begin{equation*}
	H(\theta r) \le \frac{1}{2}H(r) + C\big( \e/r+\zeta(r,\e/r)\big)^\sigma \Phi(40r).
	\end{equation*}
\end{lemma}
\begin{proof}
	Observe that the triangle inequality for $H$ implies
	\begin{equation*}
	|H(t;f) - H(t;g)| \le H(t;f-g),
	\end{equation*}
	for any $f,g\in L^2(D^\e_t;\R^d)$. Applying this to $u_\e$ and $w_\e$, we obtain
	\begin{equation}\label{est.Htr.ue}
	\begin{aligned}
	H(\theta r; u_\e) & \le H(\theta r; w_\e) + H(\theta r; u_\e - w_\e) \\
	& \le \widetilde{H}(r,\theta;w_\e) + ( H(\theta r;w_\e) - \widetilde{H}(r,\theta;w_\e) ) + H(\theta r; u_\e - w_\e)\\
	& \le \frac{1}{2} \widetilde{H}(r,1;w_\e) + ( H(\theta r;w_\e) - \widetilde{H}(r,\theta;w_\e) ) + H(\theta r; u_\e - w_\e) \\
	& \qquad + C\big( (\e/r)^{1/2} + \zeta(r,\e/r) \big) \Phi(4r) \\
	& \le \frac{1}{2} H(r; u_\e) + ( H(\theta r;w_\e) - \widetilde{H}(r,\theta;w_\e) ) + \frac{1}{2}(\widetilde{H}(r,1;w_\e) - H(r;w_\e))  \\
	&\qquad + H(\theta r; u_\e - w_\e) + H(r;u_\e - w_\e)
	+ C\big( (\e/r)^{1/2} + \zeta(r,\e/r) \big) \Phi(4r).
	\end{aligned}
	\end{equation}
	
	We first estimate
	\begin{equation*}
	\begin{aligned}
	I &:= H(\theta r;w_\e) - \widetilde{H}(r,\theta;w_\e) \\
	& = \frac{1}{\theta r} \inf_{q\in \R^d} \bigg( \fint_{D^\e_{\theta r} } |w_\e - (n_{\theta r}\cdot x)q|^2 \bigg)^{1/2} - \frac{1}{\theta r}\inf_{q\in \R^d}  \bigg( \fint_{T_r^{\e,+} \cap B_{\theta r}} |w_\e - (n_r\cdot x) q|^2\bigg)^{1/2} \\
	& \le \frac{1}{\theta r} \inf_{q\in \R^d} \bigg( \fint_{D^\e_{\theta r} } |w_\e - (n_{r}\cdot x)q|^2 \bigg)^{1/2} - \frac{1}{\theta r}\inf_{q\in \R^d}  \bigg( \fint_{T_r^{\e,+} \cap B_{\theta r}} |w_\e - (n_r\cdot x) q|^2\bigg)^{1/2} \\
	& \qquad +  \frac{C|n_r - n_{\theta r}|}{\theta r} \bigg( \fint_{T_{r}^+ } |w_\e|^2 \bigg)^{1/2}.
	\end{aligned}
	\end{equation*}
	Using Lemma \ref{lem.n_r}, the Poincar\'{e} and Caccioppoli inequalities, we see that the last term in the above inequality is bounded by $C\zeta(r,\e/r) \Phi(2r) $. Now,
	observe that
	\begin{equation*}
	\inf_{q\in \R^d} \bigg( \fint_{D^\e_{\theta r} } |w_\e - (n_{r}\cdot x)q|^2 \bigg)^{1/2} \le \frac{|T_r^{\e,+}\cap B_{\theta r}|^{1/2}}{|D^\e_{\theta r}|^{1/2}} \inf_{q\in \R^d} \bigg( \fint_{T_{r}^+ \cap B_{\theta r} } |w_\e - (n_{r}\cdot x)q|^2 \bigg)^{1/2}.
	\end{equation*}
	By the definition of $T_r^{\e,+}$ and our assumption
	\begin{equation*}
	\frac{|T_r^{\e,+}\cap B_{\theta r}|^{1/2}}{|D^\e_{\theta r}|^{1/2}} = \bigg(1 + \frac{|T_r^{\e,+}\cap B_{\theta r}\setminus D^\e_{\theta r}|}{|D^\e_{\theta r}|} \bigg)^{1/2} \le 1 + C\zeta(r,\e/r)^{1/2}.
	\end{equation*}
	Consequently,
	\begin{equation*}
	I \le C\zeta(r,\e/r)^{1/2} \Phi(2r).
	\end{equation*}
	
	Next, we estimate
	\begin{equation*}
	\begin{aligned}
	II &:= \widetilde{H}(r,1;w_\e) - H(r;w_\e) \\
	& = \frac{1}{r}\inf_{q\in \R^d}  \bigg( \fint_{T_r^{\e,+}} |w_\e - (n_r\cdot x) q|^2\bigg)^{1/2} - \frac{1}{r} \inf_{q\in \R^d} \bigg( \fint_{D^\e_{r} } |w_\e - (n_{ r}\cdot x)q|^2 \bigg)^{1/2}.
	\end{aligned}
	\end{equation*}
	To this end, let $q_r$ be the vector so that
	\begin{equation*}
	H(r;w_\e) = \frac{1}{r} \bigg( \fint_{D^\e_{r} } |w_\e - (n_{ r}\cdot x)q_r|^2 \bigg)^{1/2}.
	\end{equation*}
	Recall that $|q_r| \le C\Phi(2r)$. Hence,
	\begin{equation}\label{est.we.Tr+}
	\begin{aligned}
	\frac{1}{r}\inf_{q\in \R^d}  \bigg( \fint_{T_r^{\e,+}} |w_\e - (n_r\cdot x) q|^2\bigg)^{1/2} & \le \frac{1}{r} \bigg( \fint_{T_r^{\e,+}} |w_\e - (n_r\cdot x) q_r|^2\bigg)^{1/2} \\
	& \le \frac{1}{r} \bigg( \fint_{D^\e_r} |w_\e - (n_r\cdot x) q_r|^2\bigg)^{1/2} \frac{|D^\e_r|^{1/2}}{ |T_r^{\e,+}|^{1/2} } \\
	& \qquad + \frac{1}{r} \bigg( \frac{1}{|T_r^{\e,+}|} \int_{T_r^{\e,+}\setminus D^\e_r} |w_\e - (n_r\cdot x) q_r|^2  \bigg)^{1/2}.
	\end{aligned}
	\end{equation}
	Note that $D_r^\e \subset T_r^{\e,+}$ and $T_r^{\e,+}\setminus D^\e_r \subset T_r^{\e,+}\setminus T_r^{\e,+}$. The Poincar\'{e} inequality implies
	\begin{equation*}
	\bigg( \int_{T_r^{\e,+}\setminus D^\e_r} |w_\e|^2  \bigg)^{1/2} \le Cr\zeta(r,\e/r) \bigg( \int_{T_r^{\e,+}\setminus T_r^{\e,-}} |\nabla w_\e|^2  \bigg)^{1/2}.
	\end{equation*}
	Combining this with (\ref{est.we.Tr+}) and the estimate of $|q_r|$, we obtain
	\begin{equation*}
	\frac{1}{r}\inf_{q\in \R^d}  \bigg( \fint_{T_r^{\e,+}} |w_\e - (n_r\cdot x) q|^2\bigg)^{1/2} \le \frac{1}{r} \bigg( \fint_{D^\e_r} |w_\e - (n_r\cdot x) q_r|^2\bigg)^{1/2} + C\zeta(r,\e/r) \Phi(2r),
	\end{equation*}
	which yields
	\begin{equation*}
	II \le C\zeta(r,\e/r) \Phi(2r).
	\end{equation*}
	
	Finally, we estimate $H(r;u_\e -w_\e)$, while the estimate of $H(\theta r; u_\e -w_\e)$ is similar. Indeed, using Theorem \ref{thm.improved.approx} and the Poincar\'{e} inequality, we have that for any $r\in (\e/\e_0,\e_0)$,
	\begin{equation*}
	\begin{aligned}
	H(r;u_\e -w_\e) & = \frac{1}{r} \inf_{q\in \R^d} \bigg( \fint_{D^\e_{r} } |u_\e - w_\e - (n_{ r}\cdot x)q|^2 \bigg)^{1/2} \\
	& \le \frac{1}{r} \bigg( \fint_{D^\e_{r} } |u_\e - w_\e|^2 \bigg)^{1/2} \\
	& \le C\big( \e/r + \zeta(r,\e/r)\big)^\sigma \bigg( \fint_{D^\e_{20r} } |\nabla u_\e|^2 \bigg)^{1/2} \\
	& \le C\big( \e/r + \zeta(r,\e/r)\big)^\sigma \Phi(40r),
	\end{aligned}
	\end{equation*}
	where $\sigma\in (0,1/2)$.
	Combining this with (\ref{est.Htr.ue}) and the estimates of $I$ and $II$, we obtain the desired estimate.
\end{proof}


\subsection{Iteration}
To prove the main theorem, we need an iteration lemma which generalizes \cite[Lemma 8.5]{S17}.
\begin{lemma}\label{lem.iteration}
	Suppose $\eta: (0,1]\times (0,1]\mapsto [0,1]$ is an admissible modulus.
	Let $H,\Phi,h:(0,2]\mapsto [0,\infty)$ be nonnegative functions. Suppose that there exist $\theta\in (0,1/4), \e_0\in (0,\theta)$ and $C_0>0$ so that $H,\Phi$ and $h$ satisfy:
	\begin{itemize}
		\item for every $r\in (\e/\e_0,\e_0)$,
		\begin{equation}\label{subeqn:H}
		H(\theta r) \le \frac{1}{2} H(r) + C_0 \big\{ \eta(r,\e/r)\big\} \Phi(40r)
		\end{equation}
		
		\item for every $r\in (\e,1)$,
		\begin{subequations}\label{est.HTP}
			\begin{align}
			H(r) & \le C_0 \Phi(r) \label{subeqn:a}\\	
			h(r) & \le C_0 \big( H(r) + \Phi(r) \big) \label{subeqn:b}\\
			\Phi(r) &\le C_0\big( H(r) + h(r)\big) \label{subeqn:c}\\
			\sup_{r\le t\le 2r} \Phi(t) & \le C_0\Phi(2r) \label{subeqn:d}\\
			\sup_{r\le s,t\le 2r} |h(s) - h(t)| &\le C_0 H(2r) \label{subeqn:e}
			\end{align}
		\end{subequations}
	\end{itemize}
	Then
	\begin{equation}\label{est.iteration}
	\int_{\e}^{1} \frac{H(r)}{r} dr + \sup_{\e\le r\le 1} \Phi(r) \le C\Phi(2),
	\end{equation}
	where $C$ depends on the parameters except $\e$.
\end{lemma}

\begin{proof}
	We start from an estimate of $h$. The assumption (\ref{subeqn:e}) on $h$ implies $h(r) \le h(2r) + CH(2r)$. Hence, given any $t\in (\e,1)$
	\begin{equation*}
	\begin{aligned}
	\int_t^{1} \frac{h(r)}{r} dr & \le \int_t^{1} \frac{h(2r)}{r} dr + C_0\int_t^{1} \frac{H(2r)}{r} dr \\
	& = \int_{2t}^{2} \frac{h(r)}{r} dr + C_0\int_{2t}^2 \frac{H(r)}{r} dr.
	\end{aligned}
	\end{equation*}
	It follows from (\ref{subeqn:b}), (\ref{subeqn:a}) and (\ref{subeqn:d}) in sequence that
	\begin{equation*}
	\int_{t}^{2t} \frac{h(r)}{r} dr \le C\Phi(2) + C\int_{2t}^2 \frac{H(r)}{r} dr.
	\end{equation*}
	Hence, by using (\ref{subeqn:e}) again, for every $t\in (\e,1)$,
	\begin{equation}\label{est.ht}
	h(t) \le C\Phi(2) + C\int_{t}^2 \frac{H(r)}{r} dr.
	\end{equation}
	
	Let $\alpha\in (0,\e_0)$ be a small number to be determined. Without loss of generality, assume $\e < \alpha^2 \le \e_0^2$. Integrating (\ref{subeqn:H}) over the interval $[\e/\alpha,\alpha]\subset [\e/\e_0,\e_0]$, we have
	\begin{equation}\label{est.H/r}
	\int_{\e/\alpha}^{\alpha} \frac{H(\theta r)}{r} dr \le \frac{1}{2}\int_{\e/\alpha}^{\alpha} \frac{H(r)}{r} dr + C_0 \int_{\e/\alpha}^{\alpha} \eta(r,\e/r) \Phi(40r) \frac{dr}{r}.
	\end{equation}
	Using the condition (\ref{subeqn:c}), we have
	\begin{equation*}
	\begin{aligned}
	\int_{\e/\alpha}^{\alpha} \eta(r,\e/r) \Phi(40r)\frac{dr}{r} \le C_0 \int_{\e/\alpha}^{\alpha} \eta(r,\e/r) (H(40r) + h(40r)) \frac{dr}{r}.
	\end{aligned}
	\end{equation*}
	Now, we observe that (\ref{est.ht}) implies
	\begin{equation*}
	\begin{aligned}
	\int_{\e/\alpha}^{\alpha} \eta(r,\e/r) H(40r) \frac{dr}{r}& \le \Big( \sup_{r,s\in (0,\alpha)} \eta(r,s) \Big) \int_{\e/\alpha}^{\alpha} H(40r) \frac{dr}{r} \\
	&= \Big( \sup_{r,s\in (0,\alpha)} \eta(r,s) \Big) \int_{40\e/\alpha}^{40\alpha} H(r) \frac{dr}{r},
	\end{aligned}
	\end{equation*}
	and
	\begin{equation*}
	\begin{aligned}
	& \int_{\e/\alpha}^{\alpha} \eta(r,\e/r) h(40r) \frac{dr}{r} \\
	&\qquad \le C\Phi(2) \int_{\e/\alpha}^{\alpha} \eta(r,\e/r) \frac{dr}{r} + C\int_{\e/\alpha}^{\alpha} \eta(r,\e/r) \int_{40r}^2 \frac{H(s)}{s} ds dr.
	\end{aligned}
	\end{equation*}
	Combining the last four inequalities, we obtain
	\begin{equation}\label{est.Hr.theta}
	\begin{aligned}
	& \int_{\theta \e/\alpha}^{\theta\alpha} \frac{H( r)}{r} dr \\
	&\le C\Phi(2) + \bigg[ \frac{1}{2} + C\bigg\{ \sup_{r,s\in (0,\alpha)} \eta(r,s) + \sup_{\e\in (0,\alpha^2)} \int_{\e/\alpha}^{\alpha} \eta(r,\e/r) \frac{dr}{r} \bigg\} \bigg]\int_{\e/\alpha}^2 \frac{H(r)}{r} dr.
	\end{aligned}
	\end{equation}
	
	Now, by the hypothesis of $\eta$, namely, $\eta$ is an admissible modulus, we may choose and fix an $\alpha$ sufficiently small so that
	\begin{equation*}
	\frac{1}{2} + C\bigg\{ \sup_{r,s\in (0,\alpha)} \eta(r,s) + \sup_{\e\in (0,\alpha^2)} \int_{\e/\alpha}^{\alpha} \eta(r,\e/r) \frac{dr}{r} \bigg\} \le \frac{3}{4}.
	\end{equation*}
	It is quite important to note that $\alpha$ is independent of $\e$.
	Consequently, it follows from (\ref{est.Hr.theta}) that
	\begin{equation}\label{est.Hr.ab}
	\begin{aligned}
	\int_{\theta \e/\alpha}^{2} \frac{H( r)}{r} dr &\le C\Phi(2) + 3 \int_{\theta\alpha}^2 \frac{H(r)}{r} dr\le C\Phi(2),
	\end{aligned}
	\end{equation}
	where we also used (\ref{subeqn:a}) and (\ref{subeqn:d}) in the last inequality. Of course, the constant $C$ above depends on $\theta$ and $\alpha$. This is harmless since they are fixed constants independent of $\e$. In view of (\ref{est.ht}), this gives
	\begin{equation}\label{est.hr}
	h(r) \le C\Phi(2), \qquad \text{for any } r\in (\theta\e/\alpha, 2).
	\end{equation}
	Therefore, for any $t\in (\theta\e/\alpha, 2)$, by (\ref{subeqn:c}), (\ref{est.Hr.ab}) and (\ref{est.hr}),
	\begin{equation*}
	\int_t^{2t} \frac{\Phi(r)}{r} dr \le C_0\int_t^{2t} \frac{H(r)}{r} dr + C_0 \int_t^{2t} \frac{h(r)}{r} dr \le C\Phi(2).
	\end{equation*}
	In view of (\ref{subeqn:d}), this implies that
	\begin{equation}\label{est.Phir}
	\Phi(r) \le C\Phi(2), \qquad \text{for any } r\in (\theta\e/\alpha, 2).
	\end{equation}
	Note that (\ref{est.Hr.ab}) and (\ref{est.Phir}) almost give the desired estimate (\ref{est.iteration}), except for the uncovered interval $(\e,\theta\e/\alpha)$. However, since $\theta /\alpha$ is a fixed number depending only on $C_0, \e_0$ and $\eta$, by repeatedly using (\ref{subeqn:d}) finitely many times, we recover the estimate (\ref{est.Phir}) for $r\in (\e,\theta\e/\alpha)$. Also, using (\ref{subeqn:a}), we recover 
	\begin{equation*}
	\int_{\e}^{\theta\e/\alpha } \frac{H( r)}{r} dr \le C_0\int_{\e}^{\theta\e/\alpha } \frac{\Phi( r)}{r} dr \le C_0(\theta /\alpha-1) \sup_{r\in (\e,\theta\e/\alpha )} \Phi(r) \le C\Phi(2).
	\end{equation*}
	This completes the proof.
\end{proof}

\begin{proof}[Proof of Theorem \ref{thm.main}]
	Let $\Phi$ and $H$ be defined as in  (\ref{def.Phi}) and (\ref{def.H}). Let $h$ be given in Lemma \ref{lem.h.bdry}. Define
	\begin{equation*}
	H^*(r) = H(r) + \zeta(r,\e/r) \Phi(r).
	\end{equation*}
	Then, one sees from Lemma \ref{lem.h.bdry} and \ref{lem.H} that $\Phi, H^*$ and $h$ satisfy the hypothesis of Lemma \ref{lem.iteration} (with $H$ replaced by $H^*$) with $\eta(r,s) = s^{\sigma} + \zeta(r,s)^\sigma + \zeta(\theta r, s/\theta)$ for $r\in (\e,1)$. Now, since $\zeta$ is a $\sigma$-admissible  modulus, then $\eta(r,s)$ is an admissible modulus and Lemma \ref{lem.iteration} implies
	\begin{equation*}
	\sup_{\e\le r\le 1} \Phi(r) \le C \Phi(2).
	\end{equation*}
	Finally, the Poincar\'{e} and the Caccioppoli inequalities lead to the desired estimate.
\end{proof}

\section{Local $\e$-scale Convexity}
The $\e$-scale flat domains do not include convex domains in which the boundary Lipschitz estimate actually exists for scalar elliptic equations due to the maximum (comparison) principle. In this section we consider the domains that are ``nearly convex'' at a certain point above $\e$-scale. This particularly covers both the $\e$-scale flat domains and the convex domains.

\begin{definition}\label{def.eConvex}
	Let $D^\e$ be a domain and $0 \in \partial D^\e$. We say that $D^\e$ is $\e$-scale convex at $0$ in the neighborhood $D^\e_R$ with a modulus $\zeta$, if there exists a domain $Q^\e$, such that $D^\e_R \subset Q^\e, 0\in \partial Q^\e$ and $ Q^\e$ is $\e$-scale flat with modulus $\zeta$.
\end{definition}

Suppose $u_\e\in H^1(D^\e_2)$ is a weak solution of (\ref{eq.main}) in $D^\e_2$. Again, we extend $u_\e$ to a function in $H^1(B_2)$ by zero-extension. By \cite[Lemma 7.6]{GT01}, we know $|u_\e| \in H^1(B_2), \norm{u_\e}_{H^1(B_2)} =\norm{|u_\e|}_{H^1(B_2)} $ and particularly $|u_\e| \in H^{1/2}(\partial B_2)$. Now, let $Q^\e$ be given as in Definition \ref{def.eConvex} and $Q^\e_t = Q^\e\cap B_t(0)$. Let $w_\e$ be the weak solution of
\begin{equation}\label{eq.we.Q2}
\left\{ 
\begin{aligned}
\nabla\cdot (A^\e \nabla  w_\e) &= 0 \qquad &\text{in }& Q^\e_2,\\
w_\e &= |u_\e| \qquad &\text{on }& \partial Q^\e_2.
\end{aligned}
\right.
\end{equation}
To see the well-posedness of the above equation, we let $v_\e$ be the weak solution of
\begin{equation}
\left\{ 
\begin{aligned}
\nabla\cdot A^\e \nabla v_\e &= -\nabla \cdot A^\e\nabla (|u_\e|)  \qquad &\text{in }& Q^\e_2,\\
v_\e &= 0 \qquad &\text{on }& \partial Q^\e_2.
\end{aligned}
\right.
\end{equation}
Clearly, this $v_\e$ exists because of the Lax-Milgram theorem. Then, $w_\e = v_\e + |u_\e|$ solves (\ref{eq.we.Q2}). Moreover, the energy estimate gives $\norm{ w_\e}_{H^1(Q^\e_2)} \le C\norm{u_\e}_{H^1(D^\e_2)}$. Now, the maximal principle implies $w_\e \ge 0$ in $Q_2^\e$. Since $-w_\e = -|u_\e| \le u_\e \le |u_\e| = w_\e$ on $\partial D^\e_2$, the maximal principle yields
\begin{equation}\label{est.ue<we}
|u_\e(x)| \le w_\e(x),\quad \text{for any } x\in D^\e_{2}.
\end{equation}

Now, suppose $\partial Q^\e$ is $\e$-scale flat with a $\sigma$-admissible modulus $\zeta$ for some $\sigma\in (0,1/2)$. Note that in this case, $|Q^\e_t| \simeq |B_t\setminus Q^\e_t| \simeq t^d$. Then Theorem \ref{thm.main} implies that for any $r\in (\e,1)$
\begin{equation*}
\frac{1}{2r}\bigg( \fint_{Q^\e_{2r}} |w_\e|^2 \bigg)^{1/2} \le C \bigg( \fint_{Q^\e_{2r}} |\nabla w_\e|^2 \bigg)^{1/2} \le C \bigg( \fint_{Q^\e_{2}} |\nabla w_\e|^2 \bigg)^{1/2} \le C \bigg( \int_{D^\e_{2}} |\nabla u_\e|^2 \bigg)^{1/2},
\end{equation*}
where we have used the Poincar\'{e} inequality and the fact $|Q^\e_2| \simeq 1$. Now, in view of (\ref{est.ue<we}) and the Caccioppoli inequality, we have
\begin{equation*}
\bigg(\int_{D^\e_{r}} |\nabla u_\e|^2 \bigg)^{1/2} \le \frac{C}{r}\bigg( \int_{D^\e_{2r}} | u_\e|^2 \bigg)^{1/2} \le \frac{C}{r}\bigg( \int_{Q^\e_{2r}} |w_\e|^2 \bigg)^{1/2} \le Cr^{d/2} \bigg( \int_{D^\e_{2}} |\nabla u_\e|^2 \bigg)^{1/2}.
\end{equation*}
This proves the following theorem.

\begin{theorem}
	Suppose $D^\e$ is $\e$-scale convex at $0$ in $D^\e_2(0)$ with a $\sigma$-admissible modulus $\zeta$ for some $\sigma\in (0,1/2)$. Then, for every $r\in (\e,1)$,
	\begin{equation*}
	\bigg( \int_{D^\e_r} |\nabla u_\e|^2 \bigg)^{1/2} \le Cr^{d/2} \bigg( \int_{D^\e_2} |\nabla u_\e|^2 \bigg)^{1/2},
	\end{equation*}
	where $C$ depends only on $\zeta,d,m$ and $\Lambda$.
\end{theorem}

\section*{Conflict of Interest}
The author declares no conflict of interest.

\bibliographystyle{plain}
\bibliography{Zhuge}
\end{document}